\documentclass[reqno, intlimits, sumlimits]{amsart} 

\usepackage{amsmath, amstext, amssymb, amsthm, bbm, lmodern, enumerate,graphicx}

\usepackage{hyperref}
\hypersetup{colorlinks=true,linkcolor=black,citecolor=black,filecolor=black,urlcolor=black}
\usepackage[utf8]{inputenc}

\usepackage[left=3cm,right=3cm, top=4cm, bottom=3cm]{geometry}

	\renewcommand{\Re}{\mathrm{Re}}			
	\renewcommand{\Im}{\mathrm{Im}}			
	
    \renewcommand{\tilde}{\widetilde}
    \renewcommand{\theta}{\vartheta} 
    \renewcommand{\phi}{\varphi} 
    \renewcommand{\epsilon}{\varepsilon}

	\newcommand				{\eins}			{\mathbbm{1}}   
	\newcommand				{\norm}[1]		{\left\lVert#1\right\rVert}
	\newcommand				{\abs}[1]		{\left\lvert#1\right\rvert}

	\DeclareMathOperator	{\IC}			{\mathbb{C}}
	\DeclareMathOperator	{\IE}			{\mathbb{E}} 
	\DeclareMathOperator	{\IN}			{\mathbb{N}}
	\DeclareMathOperator	{\IP}			{\mathbb{P}}
	
	\DeclareMathOperator	{\IR}			{\mathbb{R}}
	\DeclareMathOperator	{\IZ}			{\mathbb{Z}}

\theoremstyle{plain}
\newtheorem{thm}			{Theorem}
\newtheorem{lem}	[thm]	{Lemma}
\newtheorem{cor}	[thm]	{Corollary}
\newtheorem{prop}	[thm]	{Proposition}
                     
\theoremstyle{definition}

\newtheorem{bem}	[thm]	{Remark}

\begin{document}
\title{The Wasserstein distance to the Circular Law}
\author{Jonas Jalowy}
 \address{Jonas Jalowy,  Institute for Mathematical Stochastics, University of M\"unster}
 \email{jjalowy@wwu.de}
 %\date{\today}
  \thanks{Supported by the German Research Foundation (DFG) through the SPP 2265 \emph{Random Geometric Systems}}

\begin{abstract}
We investigate the Wasserstein distance between the empirical spectral distribution of non-Hermitian random matrices and the circular law. For Ginibre matrices, we obtain an optimal rate of convergence $n^{-1/2}$ in 1-Wasserstein distance. This shows that the expected transport cost of complex eigenvalues to the uniform measure on the unit disk decays faster (due to the repulsive behaviour) compared to that of i.i.d.\ points, which is known to include a logarithmic factor.
For non-Gaussian entry distributions with finite moments, we also show that the rate of convergence nearly attains this optimal rate.
\end{abstract}
 
\maketitle

%%%%%%%%%%%%%%%%%%%%% 

\section{Introduction}\label{sec:Intro}
\noindent The study of Wasserstein distances of empirical measures goes back to the seminal work of Ajtai, Koml\'os and Tusn\'ady \cite{AKT}. Their main result states that the Wasserstein distance between two empirical measures of $n$ i.i.d.\ points on the square is of order $\sqrt{\log n/ n}$. In other words, the transport cost of the optimal matching is of this order. By now, much is known about the optimal coupling between the empirical measure of $n$ independent particles and their reference measure, see for instance \cite{Del,DY,HS13,Tala}. In particular, the new PDE approach by Ambrosio, Stra and Trevisan \cite{AST} (and their exact asymptotic of the expected Wasserstein distance) had a big impact on the recent development of the field and stimulated high activity, e.g. \cite{AGT21,Bobkov,Borda,GHO,Ledoux}. Furthermore, the research on optimal transport problems of empirical measures has advanced alongside a large variety of applications, for instance in machine learning \cite{WGAN}, internet advertising matching \cite{Internet} or statistical physics \cite{Cara1,Cara2}. 

However, for dependent particles not much seems to be known (except for empirical measures of stochastic processes, see \cite{BBBW,HuMaTre,Wang}). In this paper we study the Wasserstein distance of the empirical measure of eigenvalues, which exhibit a repulsive behaviour, see \cite[\S 15]{Mehta}. Complex eigenvalues are applied not only to model Coulomb gases \cite{Forrester,Leble,SS}, but they also appear in scattering in chaotic quantum systems \cite{FKS} or in neural networks \cite{Neural} for instance.
The empirical spectral distribution of non-Hermitian random matrices with independent entries converges to the uniform distribution on the complex disc as the size of the matrix tends to infinity. This is the so-called \emph{Circular Law} whose proof has a long and interesting history going back to the 1960's  \cite{Gin65,Girko,bai97,TV08Circular,GT10Circular,TV10Circular}. For further information we refer to the survey \cite{BC12} and the references therein. In this paper, we aim to connect both of these research areas by studying the natural question about the optimal rate of convergence in Wasserstein distance to the circular law.

We consider an $n\times n$ random matrix $X$ with \emph{empirical spectral distribution} given by
\begin{align*}
 \mu_n=\frac{1}{n}\sum_{j=1}^n \delta_{\lambda_j (X/\sqrt n)},
\end{align*}
where $\delta_\lambda$ are Dirac measures in the eigenvalues $\lambda_j$ of the matrix $ X/\sqrt n$. In this note two different classes of random matrices $X$ will play a role.

\begin{enumerate}[(i)]
 \item A (complex) \emph{Ginibre matrix} $X$ is a non-Hermitian random matrix with independent complex Gaussian entries $X_{ij}\sim\mathcal N_{\IC}(0,1)$.
 \item A non-Hermitian random $n\times n$-matrix $X$ is said to have \emph{independent entries} if $X_{ij}$ are independent complex or real random variables, and in the complex case we additionally assume $\Re X_{ij}$ and $\Im X_{ij}$ to be independent.
\end{enumerate}

Under the most general second moment assumption, the Circular law is due to Tao, Vu (and Krishnapur) \cite{TV10Circular}: if $X$ has independent entries satisfying $\IE X_{ij}=0$ and $\IE\lvert X_{ij}\rvert^2=1$, then the empirical spectral distribution $\mu_n$ converges weakly to the uniform measure on the complex plane $\mu_\infty$ having Lebesgue density $\tfrac 1 \pi\eins_{B_1(0)}$ as the matrix $n$ size grows. 

We are interested in the \emph{Wasserstein distance} between these distributions. For $1\le p<\infty$, define the $p$-Wasserstein metric between two probability measures $\mu, \nu$ on $\IC$ as
\begin{align*}
 W_p(\mu,\nu)=\Bigg( \inf_{q\in \pi(\mu,\nu)}\int_{\IC\times\IC} \abs{x-y}^pdq(x,y)\Bigg)^{1/p},
\end{align*}
where $\pi(\mu,\nu)$ is the set of all couplings between $\mu, \nu$, i.e. distributions on $\IC^2$ having marginals $\mu$ and $\nu$. The Wasserstein distance between $\mu_\infty$ and $\mu_n$ describes the cost of transporting the uniform measure to the eigenvalues. Thus, this particular semi-discrete Wasserstein distance may be of importance for an optimal quantization of the uniform measure or for understanding the transport cost of a Coulomb-gas system to the locations of particles in a crystallized structure. Since each eigenvalue is allocated to some spread out area inside the unit disk (a cell), it is also of interest from a geometric point of view, see Remark \ref{rem}. 

Our first main result provides nearly optimal rate of convergence in Wasserstein distance for a general class of non-Hermitian random matrices.

\begin{thm}\label{thm:gen}
Consider a random matrix $X$ with independent entries, which are centred, normalized and all their moments exist, i.e. $\max_{i,j\le n}\IE\vert X_{i,j}\vert ^k<\infty$ for all $k\in\IN$. Then, for all $p\in \{1\}\cup [2,\infty)$, any (small) $\epsilon>0$ and every (large) $Q>0$ 
\[\IP(W^p_p(\mu_n,\mu_\infty)\le n^{-1/2+\epsilon})\ge1-n^{-Q}\]
holds for $n$ sufficiently large.
\end{thm}

In other words, the \emph{rate of convergence} to the circular law in $1$-Wasserstein distance is $\mathcal O( n^{-1/2+\epsilon})$ with overwhelming probability. For later comparison, we want to mention that this implies the same rate of convergence in an averaged sense.

\begin{cor}\label{cor:gen}
Under the assumptions of Theorem \ref{thm:gen} there exists some constant $c>0$ such that \[\IE(W_1(\mu_n,\mu_\infty))\le c n^{-1/2+\epsilon}.\]
\end{cor}

Analogously to Theorem \ref{thm:gen}, the same rate of convergence in a uniform Kolmogorov-like distance has been obtained in \cite{GJ18Rate,Jalowy} and we are going to use those results to prove Theorem \ref{thm:gen}. 

By an explicit coupling argument it has been shown by Meckes and Meckes \cite{MM14rate} that Ginibre matrices satisfy a $1$-Wasserstein rate of convergence of order at most $\mathcal O (\sqrt{\log n}/n^{1/4})$ and a $p$-Wasserstein rate of convergence for $p>2$ of order $\mathcal O \big((\log n/n)^{1/2p}\big)$. O'Rourke and Williams generalized the $1$-Wasserstein rate up to a factor $n^\epsilon$ to independent entry matrices satisfying the same moment assumption as in Theorem \ref{thm:gen}. More precisely \cite[Theorem 1.10]{OW} states $W_1(\mu_n,\mu_\infty)\le n^{-1/4+\epsilon}$ $\IP$-a.s.. 

One may also compare these eigenvalue models to the result from Coulomb gases by Chafa\"{\i}, Hardy, Ma\"{\i}da \cite{CHM16}, who studied invariant $\beta$-ensembles with external
potential $V$ instead of independent-entry matrices. Their concentration result implies a rate of convergence to the limiting measure with density $c\Delta V$ of order $\sqrt{\log n /n}$, which coincides with the rate of convergence for independent particles mentioned in the beginning.

\begin{figure}[t]
 \includegraphics[width=.33\textwidth]{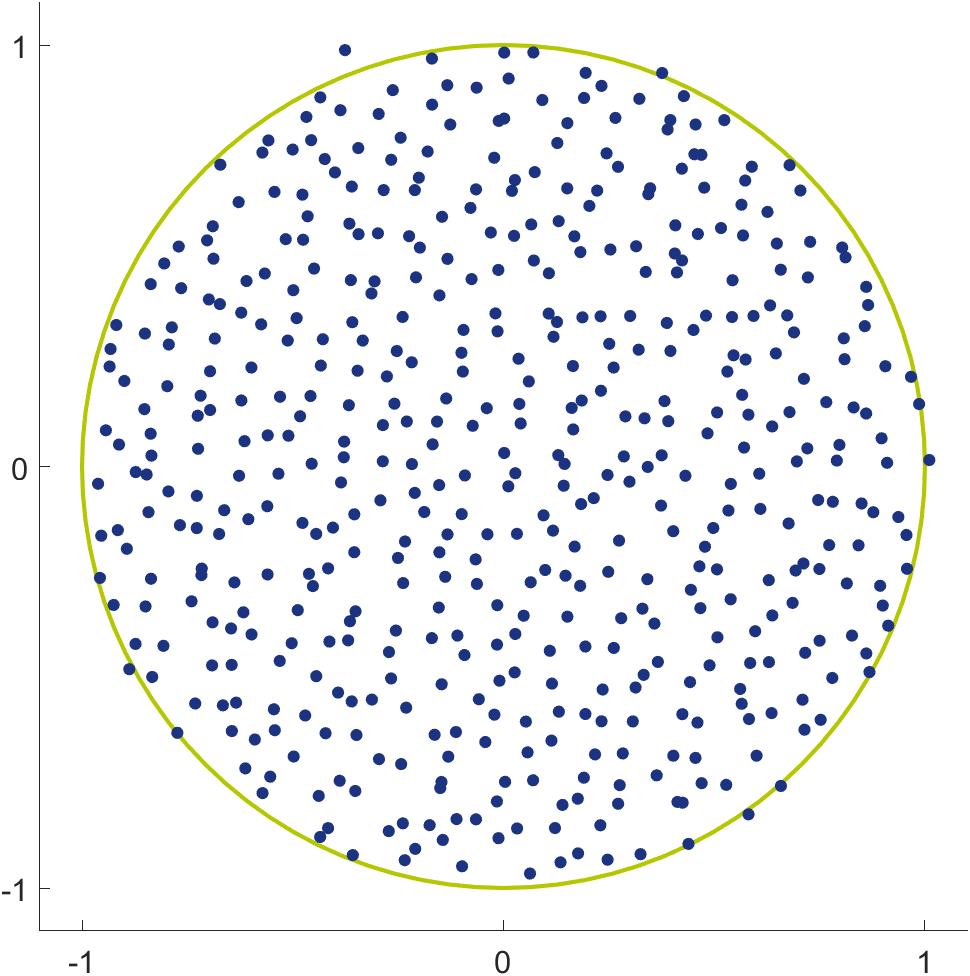}\includegraphics[width=.33\textwidth]{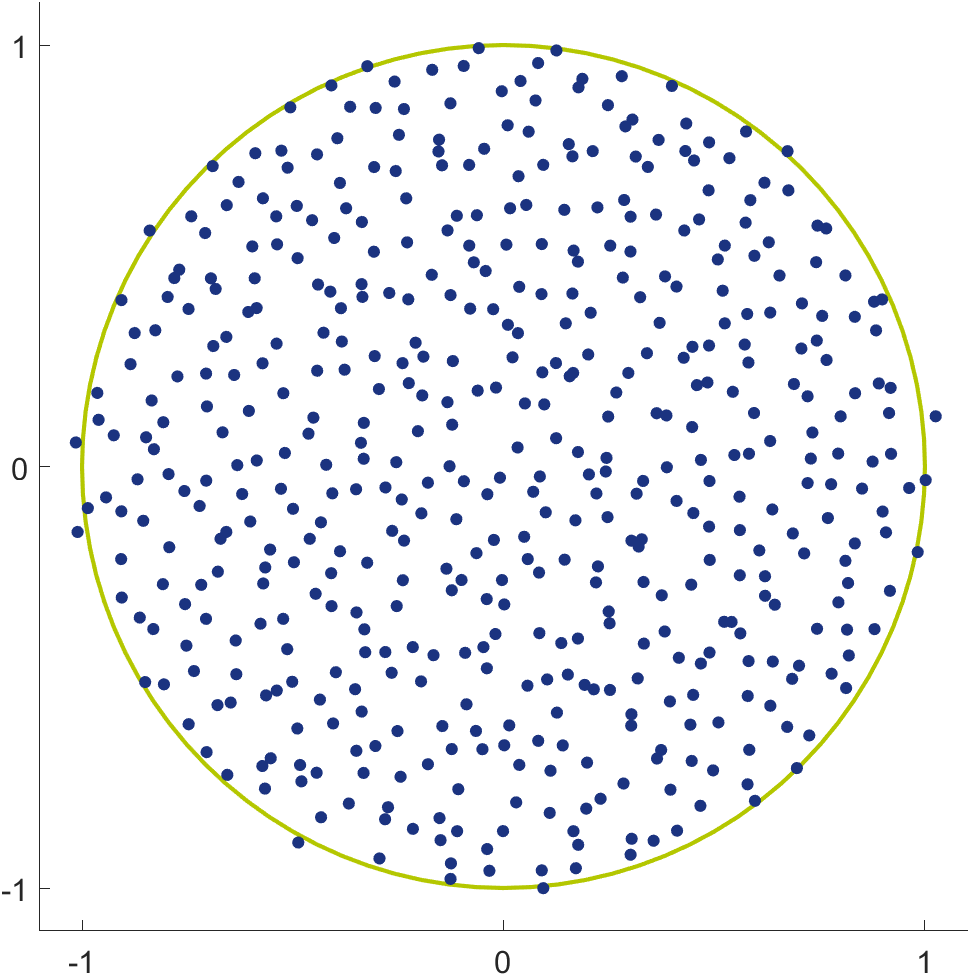}\includegraphics[width=.33\textwidth]{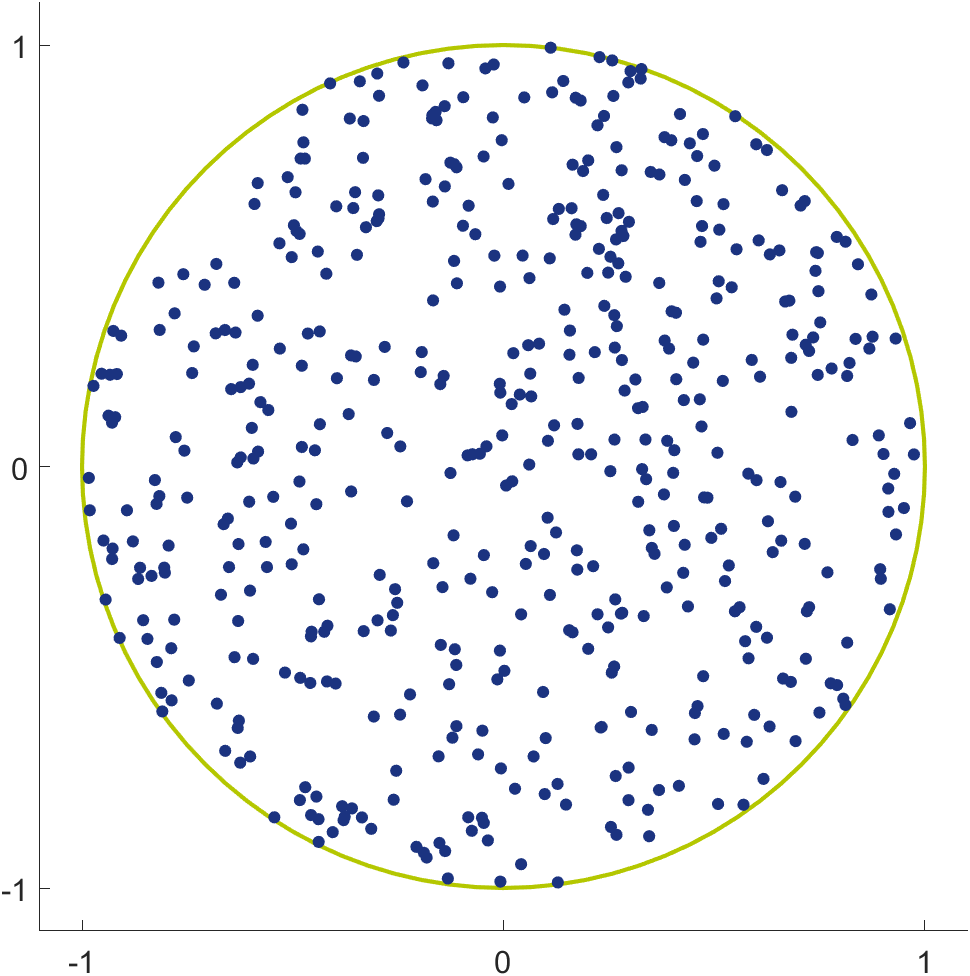}
   \caption{Samples of the empirical measure of $n=500$ points, which are eigenvalues of Ginibre matrices (left), eigenvalues of a random matrix with independent entry distributions on the discrete cube (middle) and i.i.d.\ points (right).} \label{fig:circ}
   \end{figure}
   
Already a brief look at simulations in Figure \ref{fig:circ} and Figure \ref{fig:tess} reveals that eigenvalues seem to be more uniformly distributed (due to eigenvalue repulsion) than independent points (which may have clusters or gaps), hence one may expect the transportation cost and rate of convergence of the empirical spectral distribution to the circular law to be better than $\sqrt{\log n/ n}$. Our second main result confirms this heuristic by removing the logarithmic factor and shows that the Wasserstein rate of convergence of complex eigenvalues is indeed faster than that of independent particles.

\begin{thm}\label{thm:gin}
For Ginibre matrices we have $\IE( W_1(\mu_n,\mu_\infty))\le \frac{4}{\sqrt n}$ for sufficiently large $n$.
\end{thm}

We did not aim to optimize the constant yet. The exact asymptotic for the expected 2-Wasserstein distance of i.i.d.\ points in the unit disk is $\tfrac 1 2\sqrt{\log n /n}$ according to \cite[Theorem 1.1]{AGT21}. Interestingly, the proof of Theorem \ref{thm:gin} reveals very clearly how the repulsive behaviour of the eigenvalues affects the rate of convergence and that it is the real reason for the logarithmic factor to be absent.

In the above context, Theorem \ref{thm:gin} provides the optimal rate of convergence in Wasserstein distance and Corollary \ref{cor:gen} shows that nearly the same rate is attained for non-Gaussian matrices as well - an instance of the \emph{universality phenomenon}, which is also illustrated in Figure \ref{fig:circ}. 

From a different point of view, this is the analogue of the expected rate of convergence of the empirical spectral distribution of Hermitian random matrices to the semicircular law, which is of order $1/n$ up to correction terms, see \cite{GT16,GNTT,MM2,Dalla}. Notice that the logarithmic factor still appears in the non-averaged rate of convergence and is known to be not removable due to \cite{gustav}. Thus it is unclear if the logarithmic factor will also appear in the $\IP$-a.s.\ asymptotic of $W_1(\mu_n,\mu_\infty)$.

In connection to this result, a PhD-thesis (in french) by Maxime Prod'homme should be mentioned \cite{Prod}, where a similar rate is stated for the $2$-Wasserstein distance up to an unknown constant. His approach relies on a decomposition of the unit ball into a sequence of mesoscopic boxes and a direct adoption of the methods of \cite{AST} using solutions of the Poisson equation with zero boundary terms on these boxes. This idea is similar to \cite{MM14rate,AGT21,OW} and can be seen as a constructive method.

The proofs of both of our results however use a different approach. We will work with the dual formulation of the $1$-Wasserstein distance and make use of \emph{logarithmic potentials} (a classical tool in non-Hermitian random matrix theory) which are solutions to the Poisson equation on the \emph{entire} complex plane without boundary terms. In some way, we shall find a link between the methods of \cite{AST,CHM16,GJ18Rate}.

\begin{bem}\label{rem}
Let us comment on a geometrical point of view of the results. 
Rephrasing the Wasserstein distance as a semi-discrete optimal transport problem using the dual formulation yields
\begin{align*}
W_1(\mu_n,\mu_\infty)=\max_{w\in\IR^n}\sum_{j=1}^n\int_{\mathbf C_j(w)}\big(\vert z-\lambda_j\vert -w_j\big)d\mu_\infty(z)+\frac 1 n\sum_{j=1}^n w_j,
\end{align*}
where the the Johnson-Mehl (or Apollonius) cells $\mathbf C_j(w)$ to a dual weight $w$ are defined by
\begin{align*}
\mathbf C_j(w):= \{z\in B_1(0): \vert z-\lambda_j\vert-w_j\le \vert z-\lambda_k\vert-w_k\ \forall k\neq j \}.
\end{align*}
%See \cite{PeyreCuturi} for this statement and its computational application. #
Note that $w\equiv c$ corresponds to the definition of Voronoi cells and that the maximum is attained at a dual weight $w^*$ such that the mass is equally distributed, i.e. $\mu_\infty(\mathbf C_j(w^*))=1/n$, see for instance \cite{bourne}. Moreover one can show that $\mathbf C_j(w^*)$ is a star domain with centre $\lambda_j$, in particular $\mathbf C_j$ are connected and contain $\lambda_j\in B_1(0)$. For this and more geometric aspects of the cells we refer to \cite{bourne,Bollo,HS13,voronoi} and the references therein. See also \cite{Last,naz,HHP} for different fair allocations of random points to the Lebesgue measure and \cite{Laute,GusTha} for (weighted) Laguerre tessellations. In our setting, Theorem \ref{thm:gin} can be viewed as 
\[\sum_{j=1}^n \IE\big( \int_{\mathbf C_j(w^*)}\vert \lambda_j-z\vert dz\big)= \IE\big( W_1(\mu_n,\mu_\infty)\big) \le \frac 4{\sqrt n},\]
i.e. the average $W_1$-cell of Ginibre eigenvalues is less spread out from its star-centre than that of i.i.d.\ points, cf. Figure \ref{fig:tess}.
\end{bem}

\begin{figure}[t]
 \includegraphics[width=.33\textwidth]{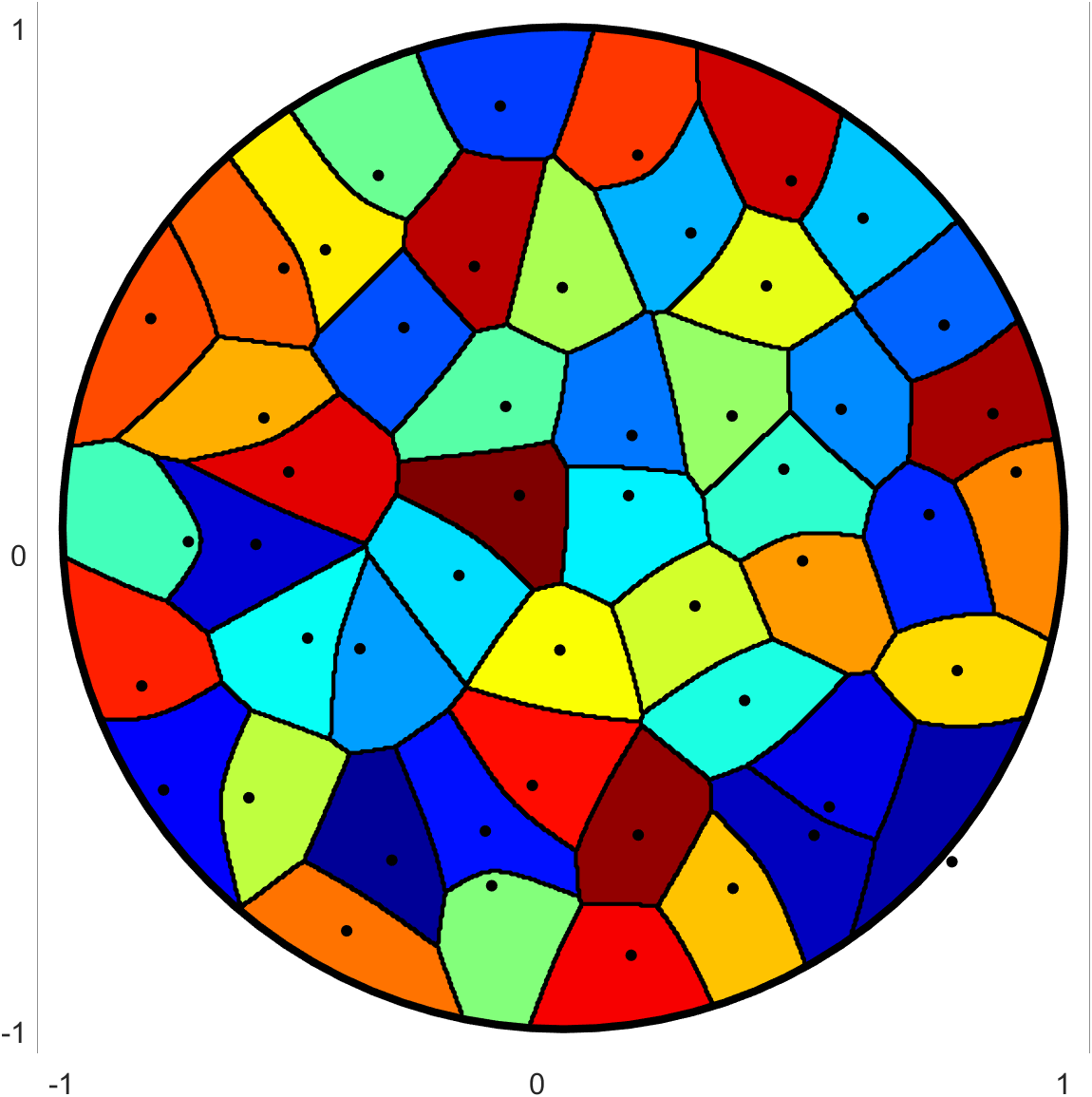}\includegraphics[width=.33\textwidth]{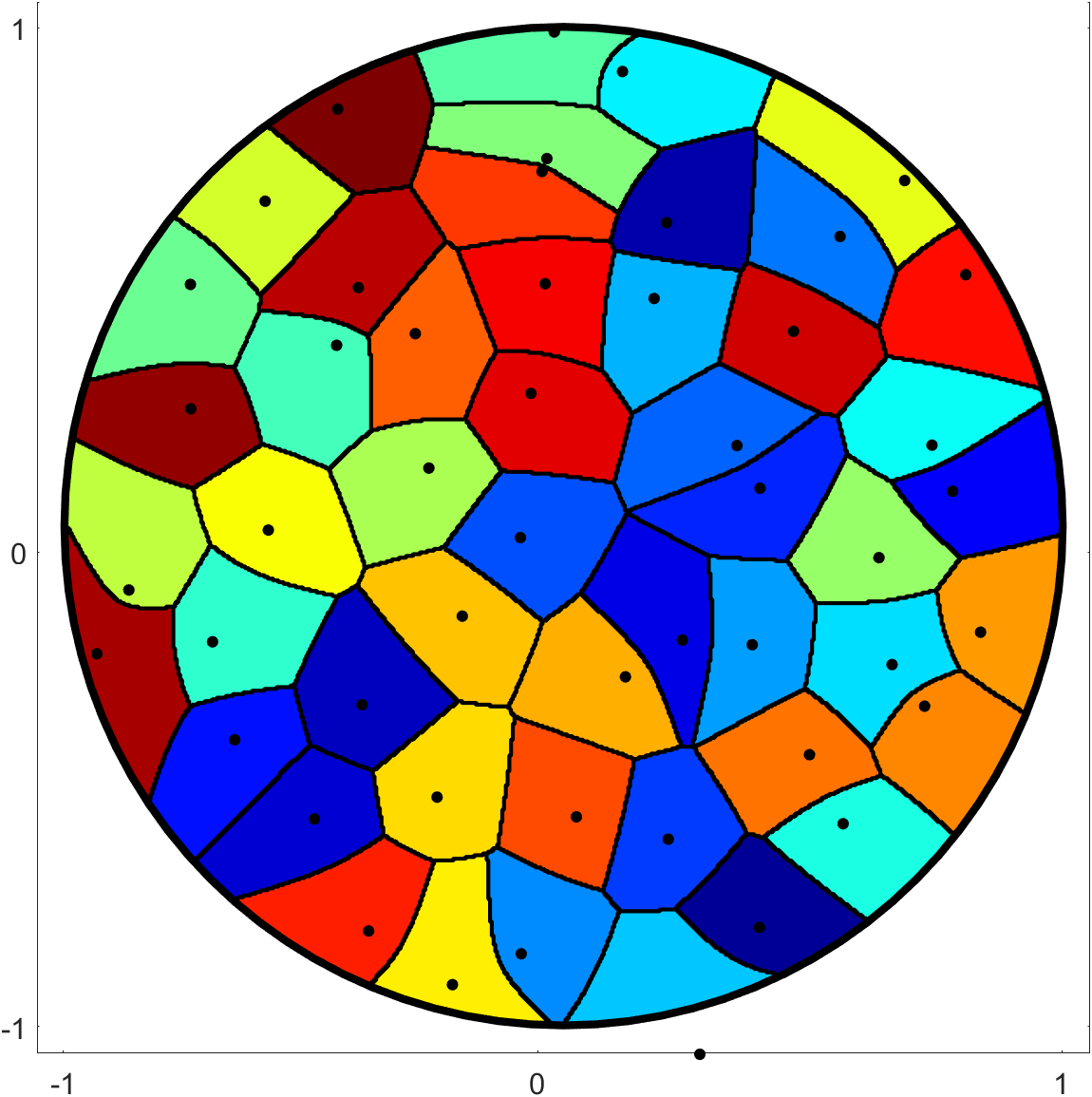}\includegraphics[width=.33\textwidth]{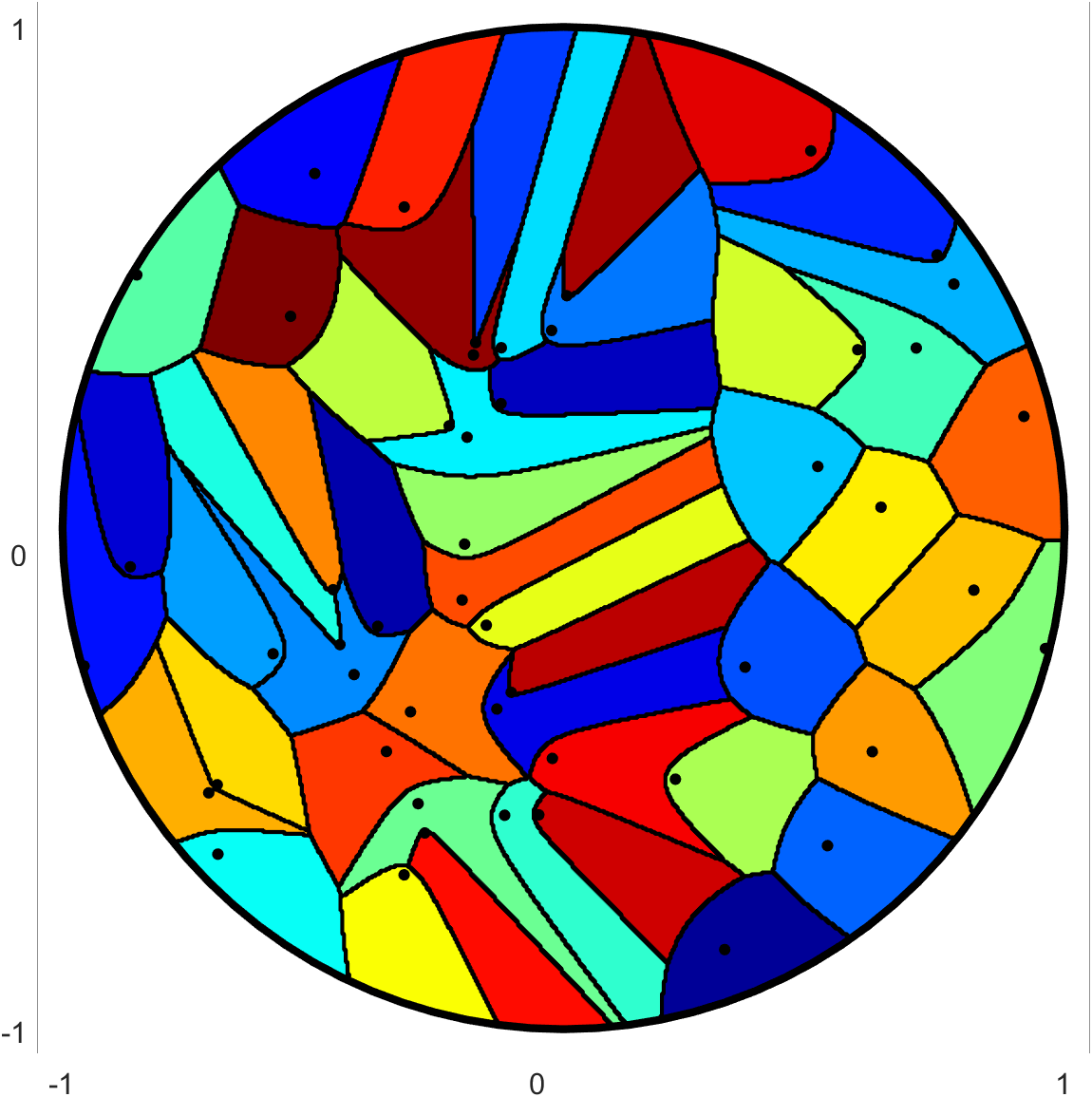}
   \caption{A visualization of the optimal transport map of $W_1(\mu_n,\mu_\infty)$ for $n=50$ points. Each point is allocated to a cell of measure $1/n$ in the (generalized) tessellation of the unit disk (the colours are arbitrary). Again, we compare eigenvalues of Ginibre matrices (left), eigenvalues of a discrete random matrix (middle) and i.i.d.\ points (right).} \label{fig:tess}
   \end{figure}

In order to show that $n^{-1/2}$ is indeed the optimal rate, we will derive an explicit lower bound for the Wasserstein distance. 

\begin{lem}\label{lem:lower}
For all $p\ge 1$ it holds $W_p(\mu_n,\mu_\infty)\ge \frac{1}{3\sqrt{n}}$.
\end{lem}
We believe that the optimal rate of convergence is of order $1/\sqrt n$ for all $p\ge 1$.

In the following $\sim$ denotes asymptotic equivalence and $\lesssim$ will denote an inequality that holds up to a $n$-independent constant $c>0$ which may differ in each occurrence. 

In a Kolmogorov-like distance, the mean empirical spectral distribution $\bar\mu_n=\IE\mu_n$ of Ginibre matrices satisfies the same rate of convergence as the non-averaged distribution. More precisely, the precise asymptotic is given by $\frac{1}{\sqrt{2\pi n}}$ (due to the difference of $(\mu_\infty-\bar\mu_n)(B_1(0))\gtrsim 1/\sqrt n$), see \cite[Lemma 1.1]{GJ18Rate}. The Wasserstein distance however turns out to converge faster if the averaged measure is considered.

 \begin{lem}\label{lem:MESD}
Ginibre matrices satisfy $ W_1(\bar\mu_n,\mu_\infty)\sim \frac{1}{2n}$.
 \end{lem}
The analogous problem for empirical measures of i.i.d. samples is trivial, since then $\bar \mu_n=\mu_\infty$ is the sample distribution. Note that Lambert \cite[Theorem 1.2]{Lam} obtained nearly the same improved rate of convergence $\log n/n$ for the non-averaged distribution by considering integral metrics with respect to a family of functions that is more regular (i.e. having essentially bounded Laplacian) rather than Lipschitz functions (having bounded gradient). %In our model of empirical spectral distributions we may also study the distance in relative entropy by using the same methods [TODO].
 
In the remaining sections, we will first prove Theorem \ref{thm:gen} for $p=1$ using the dual formulation and logarithmic potentials. Here, we also obtain the lower bound of Lemma \ref{lem:lower}. In Section 3, we prove the rate of convergence for Ginibre matrices which will make the ideas of the preceding section much more precise. After some concluding remarks, we discuss general $p\ge 2$ in Section 4 using a completely different idea.

\section{The proof of Theorem \ref{thm:gen} for $p=1$}
 
We begin with the proof for the $W_1$ convergence rate, which relies on the following approach: we will work with the dual formulation of $W_1$, regularize\footnote{The regularization procedure is also used in \cite{CHM16} for Coulomb gases and \cite{AST} for empirical measures.} $\mu_n$ and make use of the concentration of the logarithmic potentials of $\mu_n$ and $\mu_\infty$.

Define the regularized ESD $\mu_n^ r =\mu_n\star \frac 1 {2\pi r } \eins_{\partial B_ r (0)}$, which is given by the convolution of $\mu_n$ with the uniform distribution on the circle $ r \mathbb S^1=\partial B_ r (0)$.\footnote{As we will see, the particular choice of the mollification distribution is irrelevant, but this choice leads to pretty cutoffs of logarithmic singularities in the logarithmic potential.} Thus, $\mu_n^r$ is the uniform distribution on $r$-spheres around the eigenvalues $\lambda_j$, where later we shall choose $r=1/n$. In particular we have 
\begin{align}\label{eq:regularization}
W_p(\mu_n,\mu_n^ r )\le\Big(\frac 1 n\sum_{j=1}^n \frac 1 {2\pi r }\int_{\partial B_ r (\lambda_j)} \vert \lambda_j-x\vert ^p dx\Big)^{1/p}= r .
\end{align}

Recall the Kantorovich Rubinstein duality
\begin{align}\label{eq:duality} 
W_1(\mu,\nu)=\sup_{\mathrm{Lip}(f)\le 1}\int_{\IC} f(x)d(\mu-\nu)(x),
\end{align}
 where the supremum runs over all (Lipschitz-) continuous functions $f\in\mathcal C$ having Lipschitz norm $\mathrm{Lip}(f)\le 1$. Following \cite[Lemma 2.1]{CHM16}, we will localize the duality in the following way.
 
 \begin{lem}[Localization]\label{lem:Localization}
 For any $R>1$ we have with overwhelming probability
 \[W_1(\mu_n^ r ,\mu_\infty)=\sup_{\substack{\mathrm{supp} f\subseteq B_{4R}(0) \\ \mathrm{Lip}(f)\le 1}}\int_{B_{4R}(0)} f(x)d(\mu_n^ r -\mu_\infty)(x),\]
 where the supremum runs over all smooth 1-Lipschitz functions $f\in\mathcal C^\infty(B_{4R}(0))$ having support in $B_{4R}(0)$. 
 \end{lem}
 
 Here and in the following we may again restrict ourselves to the event that the spectral radius of $X/\sqrt n$ is smaller than any fixed $R>1$ with overwhelming probability (cf. \cite{Geman, BY86, AEK}), i.e. $\mathrm{supp}(\mu_n^r)\subseteq B_R(0)$. Obviously, this implies that the area of integration can be reduced to $B_R(0)$ in the above Localization, however the above representation will make integration by parts possible with vanishing boundary terms, which will be repeatedly used below.

\begin{proof}
This has been proven in \cite[Lemma 2.1]{CHM16} and for the convenience of the reader, we will include the arguments here. 

Without loss on generality, by shifting it is sufficient to consider functions with $f(0)=0$ in the Kantorovich Rubinstein duality \eqref{eq:duality}. We construct a Lipschitz cutoff $\tilde f$ such that $f=\tilde f$ on $B_R(0)$, $\mathrm{supp} \tilde f\subseteq B_{4R}(0)$ and $\mathrm{Lip}(\tilde f)\le 1$ as follows
\begin{align*}
 \tilde f(z)=\begin{cases}
              f(z) &\text{, if } \vert z\vert \le R,\\
              f\big(R \frac{z}{\vert z\vert }\big)&\text{, if } R<\vert z\vert \le 2R,\\
               f\big(R \frac{z}{\vert z\vert }\big)\frac{4R-\vert z\vert }{2R}&\text{, if } 2R<\vert z\vert \le 4R,\\
                0&\text{, if } \vert z\vert \ge 4R.\\
             \end{cases}
\end{align*}

By triangle inequality, we only have to check Lipschitz continuity between two points $z,z'$ in one and the same area. Lipschitz continuity of $\tilde f$ on $B_{2R}(0)\cup B_{4R}(0)^c$ is directly inherited from $f$. For $z,z'\in B_{4R}(0)\setminus B_{2R}(0)$, Lipschitz continuity follows from $\tilde f$ being $\tfrac 1 2 $-Lipschitz in its orthogonal angular and radial directions. More precisely for $|z'|\ge |z|$ we separate
\begin{align*}
 \vert \tilde f (z)-\tilde f(z')\vert  \le \vert \tilde f  (z)-\tilde f \big(\vert z'\vert  \tfrac{z}{\vert z\vert } \big)\vert +\vert \tilde f \big(\vert z'\vert  \tfrac{z}{\vert z\vert } \big) -f(z')\vert 
\end{align*}
where the radial term can be bounded by
\begin{align*}
 \vert \tilde f  (z)-\tilde f \big(\vert z'\vert  \tfrac{z}{\vert z\vert } \big)\vert \le \tfrac1{2R} f\big(R \tfrac{z}{\vert z\vert }\big)\abs{\vert z\vert -\vert z'\vert } \le \tfrac12\abs{z-z'},
 \end{align*}
 since $\vert f\big(R \frac{z}{\vert z\vert }\big)\vert \le R$, because of $f(0)=0$. And on the other hand the angular term is bounded by
 \begin{align*}
\vert \tilde f \big(\vert z'\vert  \tfrac{z}{\vert z\vert } \big) -f(z')\vert \le \big\lvert f\big(R \frac{z}{\vert z\vert }\big)-f\big(R \frac{z'}{\vert z'\vert }\big)\big\rvert\frac{{4R}-\vert z'\vert }{{2R}}\le \frac{\vert z\vert }{2} \abs{ \tfrac z{\vert z\vert }-\tfrac{z'}{\vert z'\vert }}\le \tfrac12\abs{z-z'}
 \end{align*}
 by Lipschitz continuity of $f$ and the orthogonal projection from $z'$ onto $\partial B_{\vert z\vert }(0)$. Therefore, $\tilde f$ is indeed $1$-Lipschitz.
 
 Now since both measures are w.o.p. supported in $B_R(0)$, we can replace an arbitrary function in \eqref{eq:duality} by
 \begin{align*}
  %\sup_{\mathrm{Lip}(f)\le 1}\int_{\IC} fd(\mu_n^r-\mu_\infty)\le
   W_1(\mu_n^r,\mu_\infty)= \sup_{\mathrm{Lip}(f)\le 1}\int_{\IC} fd(\mu_n^r-\mu_\infty) 
  =\sup_{\substack{\mathrm{Lip}(f)\le 1\\ f(0)=0}}\int_{B_{4R}(0)} \tilde fd(\mu_n^r-\mu_\infty)
  \le \sup_{\mathrm{Lip}(\tilde f)\le 1}\int_{\IC} \tilde fd(\mu_n^r-\mu_\infty).
 \end{align*}
 Hence equality follows. Ultimately, since smooth functions are dense in $\mathcal C_{Lip}(B_{4R}(0))$, it suffices to consider $f\in C^{\infty}(B_{4R}(0))$.
\end{proof}

From now on we will fix an $R>1$, to be chosen explicitly later.

Similar to the role of the Stieltjes transform in the theory of Hermitian random matrices, the weak topology of measures $\mu$ on $\IC$ can be expressed in terms of the so-called logarithmic potential $U$, which is the solution of the distributional Poisson equation. More precisely, for every compactly supported Radon measure $\mu$ on $\IC$ the \emph{logarithmic potential} is defined by
\begin{align}\label{eq:logPot}
 U_\mu(z):=-\int_{\IC}\log\abs{t-z}d\mu(t)=(-\log\abs \cdot\star\mu)(z)
\end{align}
and it satisfies
\begin{align}\label{eq:distrPoisson}
    \Delta U_\mu=-2\pi \mu 
\end{align}
in the sense of distributions. Note that this is the analogue of \cite[Equation (2.6)]{AST} for the unbounded domain $\IC$ and the homogeneous Neumann boundary condition. Let $U_n$, $U_n^ r $ and $U_\infty$ denote the logarithmic potentials of the ESD $\mu_n$, the regularized ESD $\mu_n^ r $ and the circular law $\mu_\infty$, respectively. More precisely we have
\begin{align}\label{eq:GirkoHerm}
U_n(z)=-\frac{1}{n}\sum_{j=1}^n\log\abs{\lambda_j(X/\sqrt n)-z} =-\frac{1}{n}\log\abs{\det \big( X/\sqrt n -z\big)}
\end{align}
and the log-determinant can now be rephrased as a logarithmic integral of the singular value distribution of $X/\sqrt n-z$. This relation to the ESD of a Hermitian random matrix is called \emph{Girko's Hermitization Trick}. However, we will not use this identity directly, but only implicitly when referring to \cite{AEK} below, where it plays a crucial role.

By using the mean value property of $\log\abs\cdot$, which is harmonic in $\IC\setminus\{0\}$, it can be shown that the regularized logarithmic potential appears to be a cutoff
\begin{align}\label{eq:Unr}
U_n^r(z)=&-\frac 1 n \sum_{j=1}^n \int_0^1 \log\vert z-\lambda_j+r e^{2\pi i \phi}\vert d\phi\nonumber\\
=&-\frac 1 n\sum_{j=1}^n \Big(\log(r)\eins_{B_r(\lambda_j)}(z)+\log\vert z-\lambda_j\vert \eins_{B_r(\lambda_j)^c}(z)\Big)
%=&\begin{cases}
 %             \frac{\vert J\vert }{n}\log r-\frac{1}{ n}\sum_{j\not\in J}\log\abs{z-\lambda_j} &\text{, if }\exists J\subseteq\{1,\dots,n\}: z\in\bigcap_{j\in J}B_r(\lambda_j)\\
%              U_n(z)  &\text{, if }z\not\in\bigcup_{j=1}^n B_r(\lambda_j).
%             \end{cases}
\end{align}
In particular, $U_n^r(z)=U_n(z)$ iff $z\not\in \bigcup_{j=1}^nB_r(\lambda_j)$. Similarly, it is well known that the logarithmic potential $U_\infty$ of the circular law is given by
\begin{align*}
 U_\infty(z)=\begin{cases}
              -\log\abs z &\text{, if }\abs z >1\\
              \tfrac 1 2 (1-\abs z ^2) &\text{, if }\abs z \le 1.
             \end{cases}
\end{align*}

\begin{figure}[ht]
 \includegraphics[width=.33\textwidth]{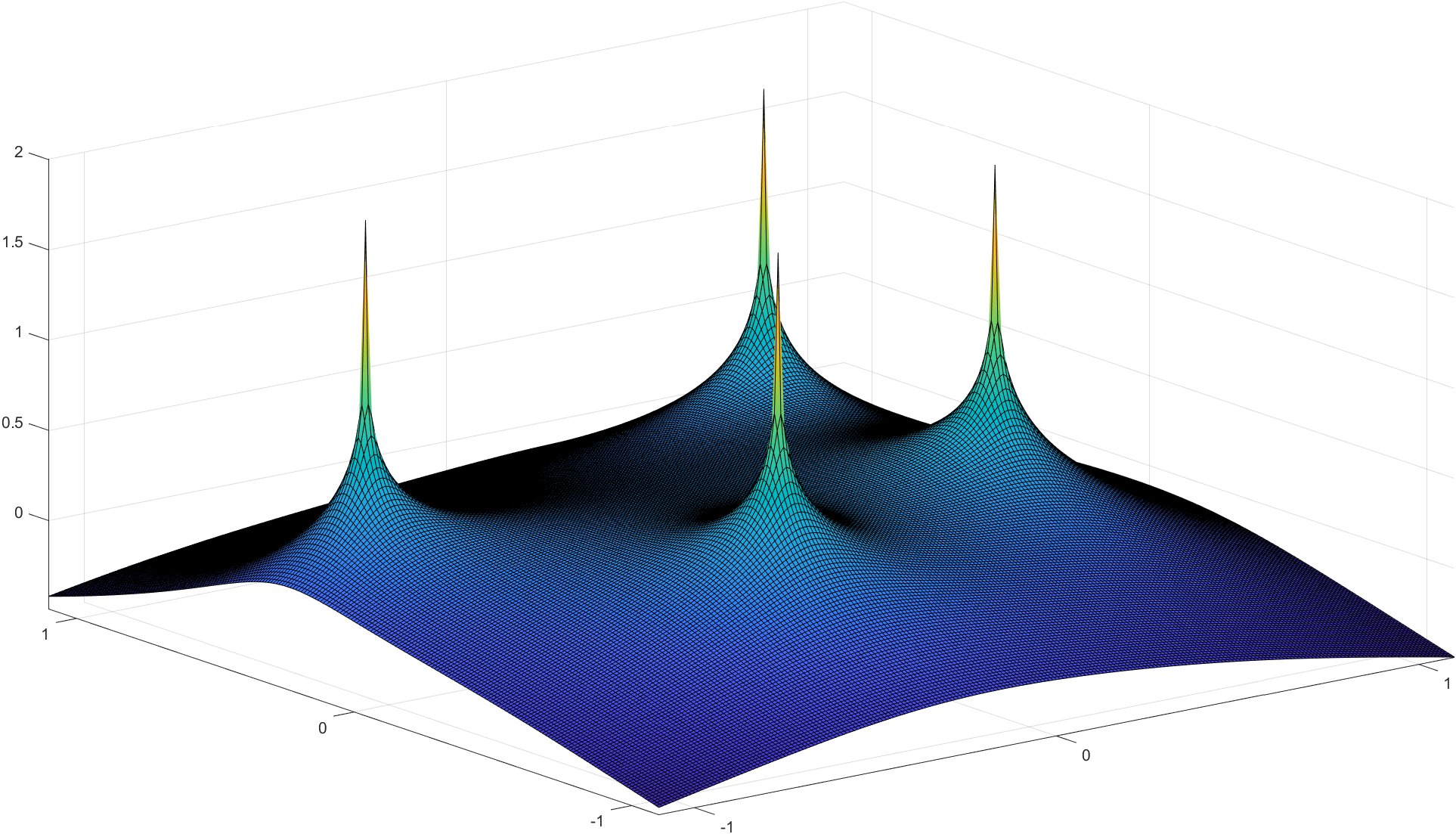}\includegraphics[width=.33\textwidth]{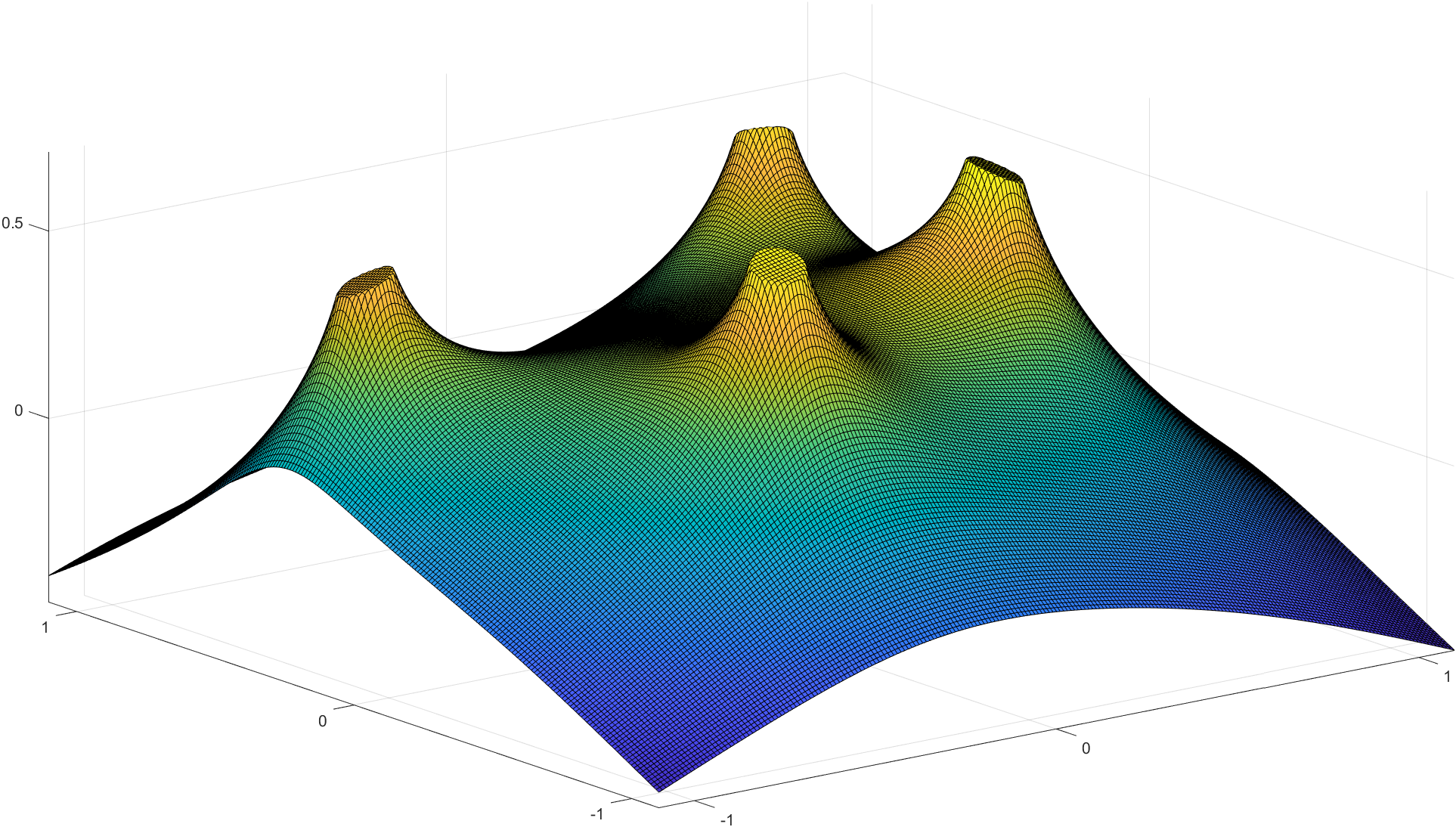}\includegraphics[width=.33\textwidth]{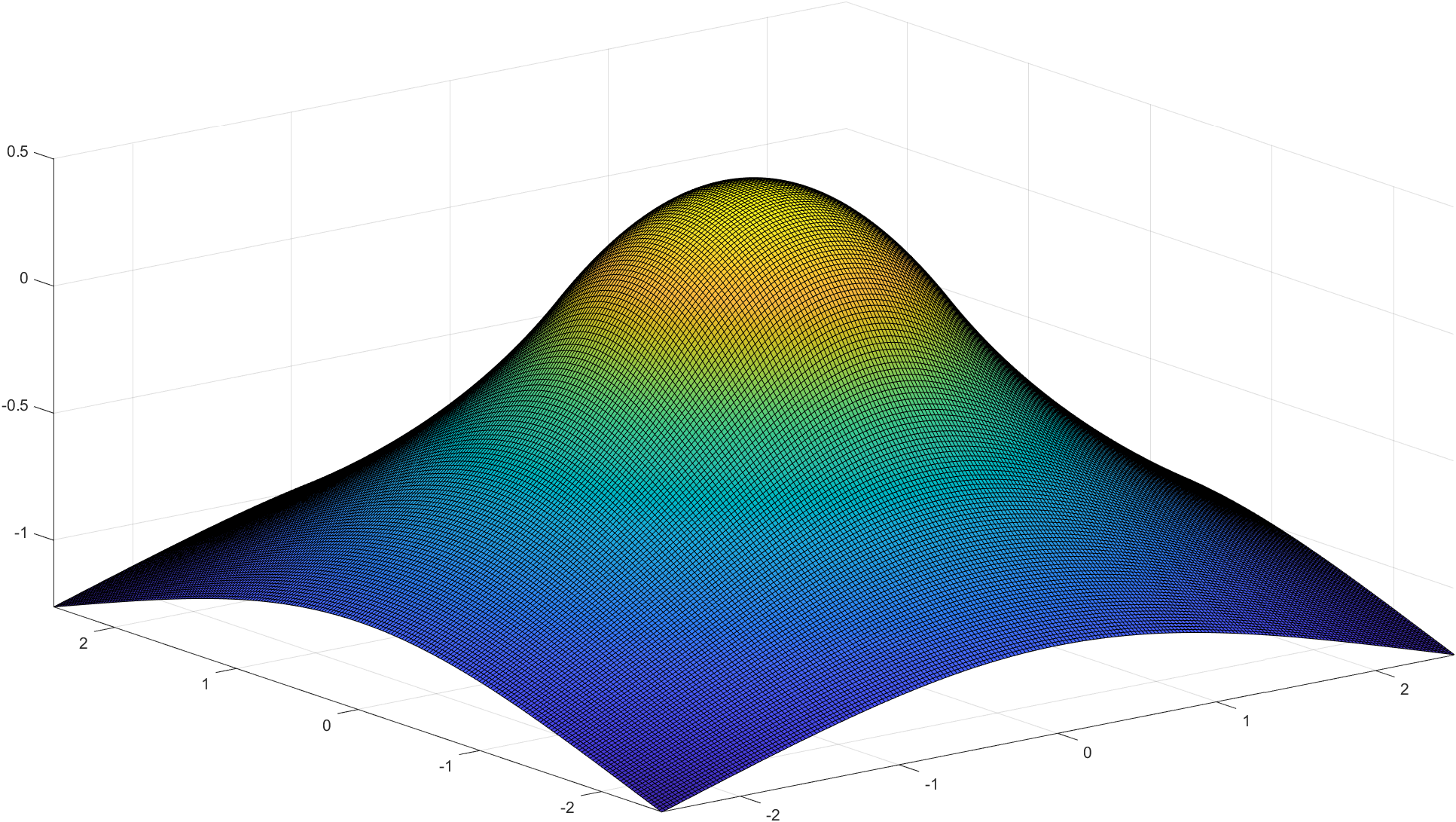}
   \caption{One sample of the logarithmic potential $U_n$ (left) for $n=4$, its 'cutoff' version $U_n^r$ (middle) and the limit function $U_\infty$ (right).}\label{fig:Un}
   \end{figure}

Under the conditions of Theorem \ref{thm:gen} it follows from the results of Alt, Erd\H{o}s, Krüger \cite{AEK} that the logarithmic potential $U_n$ concentrates around $U_\infty$, which can be made uniform for the mollified ESD $\mu_n^ r $.

\begin{prop}[Uniform concentration of the logarithmic potentials]\label{prop:ConcLogPot}
Consider a random matrix $X$ with independent entries, which are centred, normalized and all moments exist, i.e. $\max_{i,j}\IE\vert X_{i,j}\vert ^k<\infty$ for all $k\in\IN$. Then for $r=1/n$ and every $\varepsilon,Q>0$ there exists a constant $c>0$ such that 
\begin{align*}
\IP\Big(\sup_{z\in B_{4R}(0)}\abs{U_n^ r (z)-U_\infty(z)}\leq c n^{-1+\varepsilon}\Big)\geq 1-n^{-Q}.
\end{align*}
\end{prop}

Note that the analogous statement for the original logarithmic potential $U_n$ cannot hold due to the logarithmic singularities. The necessity of avoiding these singularities via two grid approximations (one random grid on circles and one deterministic grid on the disc) makes the proof the most technical part of this section.

\begin{proof}
It is an immediate consequence of the results of \cite[Lemma 6.1]{AEK} that for each fixed $z\in B_{4R}(0)$ it holds
\begin{align}\label{eq:AEK}
\IP\Big(\abs{U_n(z)-U_\infty(z)}\ge c n^{-1+\varepsilon}\Big)\le n^{-Q}.
\end{align}
for some $c>0$ (independent of $z$). The necessary arguments have been explicitly carried out in \cite[Proposition 1.1]{GJ18Rate}. In the following, we will use a random grid approximation similar to \cite[Lemma 3.2]{Jalowy} to prove a pointwise bound on $U_n^ r -U_\infty$, which can be lifted to a uniform bound due to Lipschitz-continuity. %We shall denote $a_n\lesssim b_n$ if $a_n<c b_n$ for some absolute constant $c>0$ that may change from line to line but is independent of $n$.

Let $\lambda,z\in B_{4R}(0)$ be fixed and $\xi\sim\mathcal U([0,1])$ be a random variable uniformly distributed on the unit interval and independent of our random matrix $X$, then 
\begin{align}\label{eq:lambdaclose}
 &\IP(\exists k\in\{0,\dots,n^2-1\}: \vert z-\lambda+ r  e^{2\pi i (k+\xi)/n^2}\vert \le n^{-\log n})\\
 & \le \IP(\exists k \in\{0,\dots,n^2-1\}: \vert n^2 \mathrm{arg}(\lambda-z)-2\pi (k+\xi)\vert \le 2\pi n^{3-\log n}) \le n^{3-\log n}.\nonumber
 \end{align}
In the first step we used the fact that if the points $z+re^{2\pi i (k+\xi)}$ and $\lambda$ are very close, then their angular distance on $B_r(z)$ must be close as well, up to a factor of $2\pi/r=2\pi n$.
Thus, the event of $\vert z-\lambda+ r  e^{2\pi i (k+\xi)/n^2}\vert >n^{-\log n}$ holds with overwhelming probability and in the sequel, we will restrict ourselves to this event. We begin to estimate
\begin{align}
&\abs{\int_0^1 -\log\vert z-\lambda+ r  e^{2\pi i\phi}\vert d\phi-\frac 1 {n^2}\sum_{k=0}^{n^2-1}-\log \vert  z-\lambda+ r  e^{2\pi i (k+\xi)/n^2}\vert }\nonumber\\
=&\sum_{k=0}^{n^2-1}\int_{k/n^2}^{(k+1)/n^2}\abs{\log\Big\lvert 1+ r \frac{e^{2\pi i \phi}-e^{2\pi i (k+\xi)/n^2}}{z-\lambda +  r  e^{2\pi i (k+\xi)/n^2}}\Big\rvert d\phi}\nonumber\\
\le &\frac{1}{n^2}\sum_{k=0}^{n^2-1}\log\Big( 1+\frac{2\pi r }{n^2}\frac{1}{\vert z-\lambda +  r  e^{2\pi i (k+\xi)/n^2}\vert }\Big) .\label{eq:wholesum}
\end{align}
Now split the sum into those $k$ where the second denominator exceeds or deceeds the value $n^{-2}$, i.e.
\begin{align*}
\frac{1}{n^2}\sum_{k:\vert \lambda-z- r  e^{2\pi i (k+\xi)/n^2}\vert \ge n^{-2}} \log\Big( 1+\frac{2\pi r }{n^2}\frac{1}{\vert z-\lambda +  r  e^{2\pi i (k+\xi)/n^2}\vert }\Big)\le 2\pi r.
\end{align*}
For the remaining sum, we look for the number of points in a $rn^{-2}$-fine $1$-dim grid which are $n^{-2}$-close to a fixed value. Thus, the remaining sum consists of $\mathcal O(1/r)$ summands, where each of them is bounded by
\begin{align*}
\frac{1}{n^2}\log\Big( 1+\frac{2\pi r }{n^2}\frac{1}{\vert z-\lambda +  r  e^{2\pi i (k+\xi)/n^2}\vert }\Big)\le \frac{1}{n^2}\log\Big( 1+n^{\log n}\Big)\lesssim \frac{\log(n)^2}{n^2},
\end{align*}
where we used the event introduced in \eqref{eq:lambdaclose}. Hence by choosing $r=1/n$, the total sum \eqref{eq:wholesum} is bounded by $\frac{\log(n)^2}n$. Conditioned on $X$ (i.e. freezing the eigenvalues $\lambda_j$), we conclude with overwhelming probability
\begin{align}\label{eq:Uneps}
U_n^ r (z)=\int_0^1 U_n(z+ r  e^{2\pi i\phi})d\phi=\frac 1 {n^2}\sum_{k=0}^{n^2-1}U_n( z+ r  e^{2\pi i (k+\xi)/n^2})+\mathcal O \Big(\frac{\log(n)^2}n\Big).
\end{align}
In other words, the event on the common probability space where the latter estimate holds has probability at least $1-n^{-Q}$ for all $Q>0$. On the other hand first conditioning on $\xi$ and then applying \eqref{eq:AEK} at each of the points $z+ r  e^{2\pi i(k+\xi)/n^2}$ and with $Q$ replaced by $Q+2$ yields 
\begin{align}
&\IP\Big(\exists k: \vert (U_n -U_\infty)(z+ r  e^{2\pi i(k+\xi)/n^2})\vert \ge c n^{-1+\varepsilon}\Big)\nonumber\\
&\le\sum_{k=0}^{n^2-1} \IP\Big(\vert (U_n -U_\infty)(z+ r  e^{2\pi i(k+\xi)/n^2})\vert \ge c n^{-1+\varepsilon}\Big)\le n^{-Q}\label{eq:unionbound}
\end{align}
by the union bound. Moreover, since $U_\infty$ is deterministic and regular, it follows 
\begin{align}\label{eq:Uinft}
 \vert U_\infty(z+ r  e^{2\pi i(k+\xi)/n^2})-U_\infty(z)\vert \lesssim  r =1/n.
\end{align}
Combining \eqref{eq:Uneps}, \eqref{eq:unionbound} and \eqref{eq:Uinft} implies that with overwhelming probability
\[U_n^r(z)-U_\infty(z)=\frac 1 {n^2}\sum_{k=0}^{n^2-1} (U_n-U_\infty)( z+ r  e^{2\pi i(k+\xi)/n^2} )+\mathcal O \Big( \frac{\log(n)^2}{n}\Big)=\mathcal O \big( n^{-1+\epsilon}\big). \]
In conclusion, we have shown that there is a constant $c>0$ such that for any fixed $z\in B_{4R}(0)$ and any $\epsilon,Q>0$ it holds
\begin{align}\label{eq:ConcLogPotmoll}
\IP\Big(\abs{U_n^r(z) -U_\infty(z)}\ge c n^{-1+\varepsilon}\Big)\le n^{-Q}.
\end{align}
In the final step, we choose a grid\footnote{This grid does not need to be randomly shifted because $U_n^r$ does not have singularities anymore} $B_{4R}(0)\cap n^{-2}\IZ^2$ and denote an enumeration of the points by $z_i$. The mollified logarithmic potential $U_n^r$ is Lipschitz continuous with constant $1/r=n$ and $U_\infty$ is 1-Lipschitz. Therefore we have for all $z\in B_{4R}(0)$
\begin{align*}
 U_n^r(z)-U_\infty(z)=U_n^r(z_i)-U_\infty(z_i)+\mathcal O \big(1/n),
\end{align*}
if we choose $z_i\in B_{4R}(0)\cap n^{-2}\IZ^2$ such that $\vert z-z_i\vert \lesssim n^{-2}$. Thus again by the union bound over all $z_i$ of the event \eqref{eq:ConcLogPotmoll} for $Q+4$ we conclude
\begin{align*}
&\IP\Big(\sup_{z\in B_{4R}(0)}\abs{U_n^ r (z)-U_\infty(z)}\gtrsim n^{-1+\varepsilon}\Big)\\
&\le \sum_{z_i}\IP\Big(\abs{U_n^ r (z_i)-U_\infty(z_i)}\gtrsim  n^{-1+\varepsilon}\Big)
\le n^{-Q}.
\end{align*}
\end{proof}

Now, we collected all the required ingredients for the

\begin{proof}[Proof of Theorem \ref{thm:gen} for $p=1$]
In the following, we will repeatedly make use of the distributional Poisson equation \eqref{eq:distrPoisson} and integration by parts. In order to justify these steps, we will need to perform a second mollification $\tilde U_n^r=U_n^r\star \phi_\delta$ and $\tilde \mu_n^r=\mu_n^r\star\phi_\delta$, where $0<\delta\to 0$ in the end of the proof and $\phi_\delta(z)=\delta^{-2}\phi(z/\delta)$ for some mollifier $\phi:\IC\to\IR_+$, symmetric, compactly supported and satisfying $\int \phi dz=1$. Such mollification satisfies $W_1(\mu, \mu\star\phi_\delta)\lesssim \delta$ for all $\mu$ due to $\Vert f\star \phi_\delta -f\Vert_{\infty}\le \delta \int |x|\phi(x)dx$ for $\mathrm{Lip}(f)\le 1$.
First note that by the triangle inequality and our regularization \eqref{eq:regularization}
\begin{align*}
W_1(\mu_n,\mu_\infty)\le W_1(\tilde\mu_n^ r ,\mu_\infty)+ r +\delta .
\end{align*}
Then, by choosing a function $f\in\mathcal C^\infty(B_{4R}(0))$ in Lemma \ref{lem:Localization} and applying the distributional Poisson equation \eqref{eq:distrPoisson} it remains to control
\begin{align*}
 \int f(z)d(\tilde\mu_n^ r -\mu_\infty)(z) =& -\frac{1}{2\pi}\int f(z) \big( \Delta \tilde U_n^ r (z) - \Delta U_\infty(z)\big) dz \\
 =& \frac{1}{2\pi}\int \langle \nabla f(z) , \nabla \tilde U_n^ r (z) - \nabla U_\infty(z)\rangle dz ,
\end{align*}
where we used integration by parts in the last step (or, more precisely the definition of the distributional derivative of $U$ with the test function $f$). %Indeed, the boundary terms vanish, since $f$ is smooth and compactly supported in $B_{4R}(0)$. 
The Cauchy Schwarz inequality yields
\begin{align}\label{eq:H-1}
 \int f(z)d(\tilde \mu_n^ r -\mu_\infty)(z) \le & \frac{1}{2\pi}\int_{B_{4R}(0)} \vert  \nabla f(z)\vert  \vert \nabla (\tilde U_n^ r  -  U_\infty)(z)\vert dz\nonumber\\
 \lesssim & \Big(\int \vert \nabla (\tilde U_n^ r  -  U_\infty)(z)\vert ^2dz\Big)^{1/2}, 
 \end{align}
 since Lip$(f)\le 1$. It should be pointed out that $\nabla (\tilde U_n^ r  -  U_\infty)\in L^2(\IC)$, since 
\begin{align*}
\abs{\nabla (\tilde U_n^ r  -  U_\infty)(z)}^2
=&\iint_{B_{4R}(0)^2} \frac{\langle z-x, z-y \rangle}{\vert z-x\vert ^2\vert z-y\vert ^2}d(\tilde \mu_n^ r -\mu_\infty) (x)d(\tilde \mu_n^ r -\mu_\infty)  (y)\\
=&\iint_{B_{4R}(0)^2} \frac{1}{\vert z\vert ^2} \big(1+\mathcal O (\vert z\vert ^{-1})\big)d(\tilde \mu_n^ r -\mu_\infty) (x)d(\tilde \mu_n^ r -\mu_\infty)  (y)\\
=&\mathcal O \big(\vert z\vert ^{-3}\big)
\end{align*} 
 as $\vert z\vert \to \infty$, since $\tilde \mu_n^ r -\mu_\infty$ has compact support and total mass zero.
We take the same route back via integration by parts to obtain
\begin{align}\label{eq:ibp=2} \frac 1 \pi \int \vert \nabla (\tilde U_n^ r  -  U_\infty)(z)\vert ^2dz=2\int  \tilde U_n^ r (z) -U_\infty(z)d(\tilde \mu_n^ r  -\mu_\infty)(z).\end{align}
Using $\mathrm{supp}(\tilde \mu_n^r)\subseteq B_{4R}(0)$ w.o.p. and taking the supremum over all $f$ according to Lemma \ref{lem:Localization} implies
\begin{align}
W_1(\tilde \mu_n^ r , \mu_\infty)\lesssim &\Big( \int  \tilde U_n^ r (z) -U_\infty(z)d(\tilde \mu_n^ r  -\mu_\infty)(z)\Big)^{1/2}\label{eq:laststep}\\
\lesssim & \Big( \sup_{z\in B_{4R}(0)} \vert  \tilde U_n^ r (z) -U_\infty(z)\vert \Big)^{1/2}.\nonumber
\end{align}
By taking $\delta \to 0$ and uniform convergence of the mollification, we have
\begin{align*}
W_1( \mu_n , \mu_\infty)\lesssim \Big( \sup_{z\in B_{4R}(0)} \vert   U_n^ r (z) -U_\infty(z)\vert \Big)^{1/2}+1/n\lesssim & c n^{-1/2+\epsilon}
\end{align*}
with overwhelming probability by Proposition \ref{prop:ConcLogPot}. 
\end{proof}

After \eqref{eq:laststep}, it seems we basically used the bound $\mu_n^ r -\mu_\infty\le \mu_n^ r +\mu_\infty$, which at a first glance looks terribly rough having in mind that $\mu_n^ r $ is converging to $\mu_\infty$. However, a closer inspection reveals that the test function $U_n^ r -U_\infty$ we integrate is positive only in the neighbourhood of the eigenvalues and this is precisely the area of the signed measure $\mu_n^ r -\mu_\infty$ being positive.

For Corollary \ref{cor:gen}, we only have to show that extraordinarily large eigenvalues appear with so small probability that will not impact the transport cost.
\begin{proof}[Proof of Corollary \ref{cor:gen}] The spectral radius $\vert \lambda\vert _{\max}$ is bounded by the spectral norm $\lVert X/\sqrt n \rVert$. Moreover for any $Q>0$ it holds
\begin{align}\label{eq:s_max}
 \IP(\lVert X/\sqrt n \rVert\geq n^{5Q})\leq \frac {\IE\norm{X/\sqrt n}^2} {n^{10Q}}\leq \frac 1 {n^{10Q+1}}\sum_{i,j=1}^{n}\IE\abs{X_{ij}}^2\le n^{-10Q+1},
\end{align}
where the spectral norm $\norm\cdot$ has been estimated by the Hilbert Schmidt norm.
Let $\Omega_1=\{R <\vert \lambda\vert _{\max}\le n^{10} \}$ and $\Omega_Q=\{n^{5Q}<\vert \lambda\vert _{\max}\le n^{5(Q+1)} \}$ for $Q\ge 2$. Hence on the event $\Omega_Q$ for each $Q\ge 2$, the whole mass of $\mu_n$ is transported at most $n^{5(Q+1)}$ far to $\mu_\infty$ and this implies already
\begin{align*}
\IE \big(W_1(\mu_n,\mu_\infty)\eins_{\Omega_ Q}\big)\le n^{5Q+5} n^{-10Q+1}= n^{-5Q+6}
\end{align*}
for $Q\ge 2$ by \eqref{eq:s_max}. On the other hand the spectral radius is smaller than $R$ w.o.p. (see again \cite{AEK} for instance), hence by choosing the complementary event to have probability $n^{-13}$ we also have
\begin{align*}
\IE \big(W_1(\mu_n,\mu_\infty)\eins_{\Omega_ 1}\big)\le n^{10} n^{-13}= n^{-3}.
\end{align*}
Altogether it follows from the event $\Omega_0$ on which Theorem \ref{thm:gen} holds with probability $1-n^{-1}$
\begin{align*}
\IE \big(W_1(\mu_n,\mu_\infty)\big)&=\IE \big(W_1(\mu_n,\mu_\infty)\eins_{\Omega_0}\big)+Rn^{-1}+n^{-3}+\sum_{Q=2}^\infty n^{-5Q+6}\lesssim n^{-1/2+\epsilon}.
\end{align*}
Note that the factor $n^{\epsilon}$ dominates any constant $c>0$, thus the bound can be either $n^{\epsilon-1/2}$ for sufficiently large $n$ or $cn^{\epsilon-1/2}$ for all $n$.
\end{proof}

The lower bound is obtained by choosing a particular function in the dual formulation.

\begin{proof}[Proof of Lemma \ref{lem:lower}]
Since $W_p\ge W_1$ for $p\ge 1$, it is sufficient to consider $p=1$. 
In the dual formulation \eqref{eq:duality}, for some $c>0$ to be chosen later, define 
 \[f(x)=\max_{j=1,\dots,n} \Big(\frac c{\sqrt n }-\vert x-\lambda_j\vert  \Big)_+ ,\]
 where $(\cdot)_+$ denotes the positive part and $\lambda_j$ are the eigenvalues of the random matrix. Obviously $f$ satisfies $\mathrm{Lip}(f)=1$ and it holds $\int fd\mu_n=c/\sqrt n$. The limiting integral is bounded by
 \[\int_{\IC} fd\mu_{\infty}\le n\Big( \big(\frac c {\sqrt n}\big)^3-\frac 1 \pi\int_{B_{c/\sqrt n}(0)}|x|dx\Big)=\frac {c^3} {3\sqrt n}.
 % \frac{c}{\sqrt n} n \Big(\frac{c}{\sqrt n}\Big)^2= c^3/\sqrt n.
 \]
 Therefore, choosing the optimal $c=1$ implies $W_1(\mu_n,\mu_\infty)\ge \int fd(\mu_n-\mu_\infty)\ge  1/3\sqrt n$.
\end{proof}

\section{Ginibre matrices}
For Ginibre matrices, much more is known about the eigenvalue distribution. Since \cite{Gin65}, the density $p_n^{(1)}$ of $\bar\mu_n=\IE\mu_n$ is known to be
\begin{align}\label{eq:DensityGinibre}
p_n^{(1)}(z)=\frac 1 \pi e^{-n\abs z ^2} \sum_{k=0}^{n-1}\frac{n^k\abs z ^{2k}}{k!},
\end{align}
which converges to $p_\infty(z)=\frac 1 \pi\eins_{B_1(0)}(z)$.
Because of the unitary invariance of the Gaussian measure on the space of matrices, the eigenvalues form a determinantal point process and hence even the correlation functions are explicit
\begin{align*}
\varrho(\sqrt n\lambda_1,\dots,\sqrt n\lambda_k)=\exp\Big[-\sum_{j=1}^kn\vert \lambda_j\vert ^2\Big]\det[K_n(\sqrt n \lambda_j,\sqrt n\lambda_k)]_{1\le j,k\le k}
\end{align*}
for the kernel $K_n(z,t)=\sum_{k=0}^{n-1}\frac{(z\bar t)^k}{\pi k!}$, see \cite{Mehta}. We will need the marginal distribution of two randomly chosen eigenvalues, which after normalization are given by

\begin{align}\label{eq:marginalGinibre}
&p_n^{(2)}(\lambda_1,\lambda_2)\\
&=\frac{n}{(n-1)\pi^2}e^{-n(\abs{\lambda_1} ^2+\abs{\lambda_2} ^2)} \Big(\sum_{k=0}^{n-1}\frac{n^k\vert \lambda_1\vert ^{2k}}{k!}\sum_{k=0}^{n-1}\frac{n^k\vert \lambda_2\vert ^{2k}}{k!}-\Big\lvert\sum_{k=0}^{n-1}\frac{n^k(\lambda_1\bar\lambda_2)^{k}}{k!}\Big\rvert^2 \Big)\nonumber
\end{align}

It should be remarked that the marginal density is the difference of two terms, where the first one corresponds to independent particles having density $p_n^{(1)}(\lambda_1)p_n^{(1)}(\lambda_2)$ and the second term describes the correlation between them: the bigger the latter it is, the more dependency (in this case repulsion) between the particles. This insight will play a crucial role in the proof below.

It is well known that the incomplete exponential series appearing in the densities can be estimated as follows, depending on the radius.
\begin{lem}\label{lem:exp}
For $\rho \le 1$ it holds
\begin{align}\label{eq:exple1}
\sum_{k=n}^\infty \frac{(n\rho^2)^k}{k!}\le \frac{\rho ^{2n}e^n}{\sqrt{2\pi n}}\frac{1+1/n}{1-\rho^2 +1/n}
\end{align}
and similarly for $\rho \ge 1$ it holds
\begin{align}\label{eq:expge1}
\sum_{k=0}^{n-1}\frac{(n\rho^2)^k}{k!}\le \frac{\rho ^{2n}e^n}{\sqrt{2\pi n}}\frac{1}{\rho^2 -1+1/n}.
\end{align}
\end{lem}
\begin{proof}
For $\rho \le 1$ it holds
\begin{align*}
 \sum_{k=n}^{\infty}\frac{n^k\rho  ^{2k}}{k!}&\leq  \frac{(n\rho^2  )^n}{ n!}\sum_{k=0}^\infty\left( \frac{n\rho ^2 }{n+1}\right)^k\\
&= \frac{(n\rho^2  )^n}{n!} \frac{n+1}{n(1-\rho^2   )+1}\le\frac{\rho ^{2n}e^n}{\sqrt{2\pi n}}\frac{1+1/n}{1-\rho^2 +1/n}
\end{align*}
by Stirling's formula. On the other hand if $\rho \geq 1$, then we have analogously
\begin{align*}
\sum_{k=0}^{n-1} \frac{(n\rho^2  )^k }{k!}&\le\frac{(n\rho^2 )^{n-1}}{(n-1)!}\sum_{k=0}^{n-1}\left(\frac{n-1}{n\rho^2   } \right)^k\\
&\le \frac{(n\rho^2  )^{n}}{n!} \frac{1}{\rho ^2 -(n-1)/n}\le \frac{\rho ^{2n}e^n}{\sqrt{2\pi n}}\frac{1}{\rho^2 -1+1/n}.
\end{align*}
\end{proof}

In order to prove Theorem \ref{thm:gin}, it will be helpful to first apply to Lemma \ref{lem:MESD} in order to compare $\mu_n$ to its average $\bar\mu_n$ instead of its limit $\mu_\infty$, since $\bar\mu_n$ already captures the average distribution of eigenvalues which need to be transported. The following proof presents some basic concepts and serves as a warm-up for the proof of Theorem \ref{thm:gin}.

\begin{proof}[Proof of Lemma \ref{lem:MESD}]
 Again, we work on the dual side of the problem and consider 
\[\int_{\IC} f(z) d(\bar\mu_n-\mu_\infty)(z)=2\int_0^\infty f(\rho )\Big(e^{-n\rho ^2}\sum_{k=0}^{n-1}\frac{n^k\rho ^{2k}}{k!}-\eins_{\rho <1}\Big)\rho d\rho .\]
Here we restricted ourselves to rotationally symmetric $\mathcal C^1$-Lipschitz functions $f(z)=f(\vert z\vert )$, since $\bar\mu_n-\mu_\infty$ is a rotationally symmetric measure.\footnote{From a non-dual point of view, the optimal coupling is given by a radially monotone rearrangement.} Let us denote 
\begin{align*}
\bar D_n(\rho )=(\mu_\infty-\bar\mu_n)(B_\rho (0))=\begin{cases}
                                                e^{-n\rho^2}\Big(\frac{(n\rho^2)^n}{n!}-(1-\rho^2)\sum_{k=n}^\infty\frac{(n\rho^2)^k}{k!}\Big)&\text{, if }\rho\le 1,\\ e^{-n\rho^2}\Big(\frac{(n\rho^2)^n}{n!}-(\rho^2-1)\sum_{k=0}^{n-1}\frac{(n\rho^2)^k}{k!}\Big)&\text{, if }\rho\ge 1
                                               \end{cases}
\end{align*} which is a weak antiderivative of $2\rho \Big(\eins_{\rho <1}-e^{-n\rho ^2}\sum_{k=0}^{n-1}\frac{n^k\rho ^{2k}}{k!}\Big)$, see \cite[Proof of Lemma 1.1]{GJ18Rate} for details. Using integration by parts with $\bar D_n(0)=\bar D_n(\infty)=0$, we obtain 
\begin{align}\label{eq:ibp}
\int f d(\bar\mu_n-\mu_\infty)=\int_0^\infty f'(\rho ) \bar D_n(\rho )d\rho .
 \end{align}

For the upper bound, we will use the uniform bound 
\[\bar D_n(\rho )\le e^{-n\rho ^2}\frac{(n\rho ^2)^n}{n!}\le\frac 1{\sqrt{2\pi n}}e^{n}e^{-n\rho ^2}\rho ^{2n}=\frac{\exp\big(-n(\rho^2 -1-\log(\rho^2) )\big) }{\sqrt{2\pi n}}.\]
 Hence we obtain
\begin{align*}
 W_1(\bar\mu_n,\mu_\infty)\le \sup_{\mathrm{Lip}(f)\le 1}\int_0^\infty \vert f'(\rho )\vert  \vert \bar D_n(\rho )\vert d\rho  \le \frac{1}{2\sqrt{2\pi n}}e^{n}\int_{-\infty}^\infty e^{-n\rho ^2}\rho ^{2n}d\rho ,
\end{align*}
which can be expressed as the $2n$'th moment of $\mathcal N(0,1/(2n))$, i.e.
\[\frac{1}{\sqrt{2\pi/n}}e^n\int_{-\infty}^\infty e^{-n\rho ^2}\rho ^{2n}d\rho =e^n\frac{1}{\sqrt 2}\frac{(2n)!}{2^n n!} \frac{1}{(2n)^n}\sim 1.\]
Thus we have shown $W_1(\bar\mu_n,\mu_\infty)\le \frac{1}{2n} (1+o(1))$. 

On the other hand it follows from Lemma \ref{lem:exp} for all $\rho>0$
\[\bar D_n(\rho )\ge e^{-n\rho ^2}\frac{(n\rho ^2)^n}{n!}\frac{\rho^2\wedge 1}{n\vert 1-\rho^2\vert +1}\sim\frac {\exp\big(-n(\rho^2-1-\log(\rho^2))\big)} {\sqrt{2\pi n}}\frac{\rho^2\wedge 1}{n\vert 1-\rho^2\vert +1}.\] 
The choice $f(\rho )=\rho $ and a simple application of Laplace Method at the minimum $0$ of $g(\rho)= \rho^2-1-\log(\rho^2)$ at $\rho=1$ yields
\[W_1(\bar\mu_n,\mu_\infty)\ge \int \bar D_n(\rho )d\rho  \sim \frac{\sqrt{2\pi}}{\sqrt{2\pi n}\sqrt{ g''(0)n}}\sim \frac{1}{2n}.\]
We omit the details, because we will do similar steps below.
\end{proof}

\subsubsection*{Idea of the proof} 
Following the idea of the proof in the previous section, we will make use of the dual formulation of the 1-Wasserstein distance, use the smoothened logarithmic potentials and carefully estimate the difference of logarithmic potentials integrated with respect to the difference of measures - very much in the spirit of \eqref{eq:laststep}. However we do not expect a more detailed estimate on the $\omega$-wise concentration of logarithmic potential to hold for Ginibre matrices, cf. \cite[Theorem 33]{TV15uni}. On the other hand our computations will be explicit by averaging first, instead of looking at estimates with overwhelming probability. In this way, we can make use of the correlation functions. 

Recall from \eqref{eq:logPot} the logarithmic potential $U_n=-\log\vert \cdot \vert \star \mu_n$ and its mean $\bar U_n=-\log\vert \cdot \vert \star \bar\mu_n$. After a short preparation, we need to estimate, formally, (cf. \eqref{eq:laststep})
\begin{align}\label{eq:formally}
\IE\big(\int (U_n-\bar U_n) d(\mu_n-\bar\mu_n)\big)&=\IE\big(\int U_n d\mu_n\big)-\int \bar U_n d\bar\mu_n\nonumber\\
&=\frac{1}{n^2}\sum_{j,k=1}^n\IE(-\log\vert \lambda_j-\lambda_k\vert )-\int \bar U_n d\bar\mu_n.
\end{align}
Of course, the diagonal terms in the above equation do not exist, thus we need to mollify the logarithmic singularity. The formal idea is then to use the distribution of two (different) eigenvalues, given by \eqref{eq:marginalGinibre}, where the "limiting term" in \eqref{eq:formally} cancels with the first term of the marginal distribution (as mentioned below \eqref{eq:marginalGinibre} it describes independent particles). The leading term of the integral of $\log\vert \lambda_1-\lambda_2\vert $ with respect to the marginal $p^{(2)}(\lambda_1,\lambda_2)$ will come from eigenvalues in the essential support of $\bar\mu_n$, which lie very close to each other (cf. \eqref{eq:leading}) - very much in the spirit of the heuristic described below \eqref{eq:marginalGinibre}. The first order term of order $\log n /n$ will be cancelled by the diagonal term for $i=j$ in \eqref{eq:formally}, causing the explosion at logarithmic singularities.\footnote{On a technical level, this is the reason for the Wasserstein distance of independent particles to include the logarithmic error term: the diagonal term is not cancelled, since the marginal distribution only consists of the first part of \eqref{eq:marginalGinibre}.} The second order term of order $1/n$ will, after taking the root as in \eqref{eq:laststep}, be the leading order term.\footnote{This can be seen as a renormalized energy, see \cite{SS} or \cite[Lemma 3.17]{AST}} All remaining integration areas will be of negligible order.\footnote{Note that the leading order term for the marginal distribution $p^{(1)}$ of a single particle comes from an area close to the boundary, which will not play an important role here because the repulsion between different eigenvalues at the edge is smaller due to a smaller density of states.}

We begin with a smoothing of \eqref{eq:formally} on two levels in the same way as in the proof of Theorem \ref{thm:gen} for $p=1$, however we will define it in one mollification step: let $\phi_\delta(z)=\delta^{-2}\phi(z/\delta)$ be a smooth mollifier for some $\phi:\IC\to\IR_+$, which is rotationally invariant, supported in $B_1(0)$ and satisfies $\int \phi dz=1$. Define $\tilde\phi_r=\phi_\delta\star\frac 1 {2\pi (r-\delta) } \eins_{\partial B_{r-\delta} (0)}$, which by definition is also a smooth mollifier supported in $B_r(0)$.  Later, we will pick $r=\kappa/\sqrt{n}$ for some suitable $\kappa>0$ and $\delta=n^{-10}$.

Furthermore let $\mu_n^ r =\mu_n\star \tilde \phi_r$ and $U_n^r=-\log\vert \cdot\vert \star\mu_n^r$ be the smoothened (on scale $\delta$) and regularized (cutoff on scale $r$) logarithmic potential. Under slight abuse of notation we will write $U_n^{2r}:=U_n\star \tilde \phi_r\star \tilde \phi_r$, as well as for instance $ \bar U_n^r,\bar U_n^{2r}$, respectively, for the mean ESD $\bar\mu_n^r=\bar\mu_n\star \tilde \phi_r$. Note that we will be able to exchange integration and convolution repeatedly in the sequel, for instance by Campbell's formula
\begin{align}\label{eq:intconv}
\IE\Big(\int f d\mu^r_n\Big)=\int f d\bar\mu^r_n=\int f\star\tilde\phi_rd\bar\mu_n.
\end{align}

\begin{proof}[Proof of Theorem \ref{thm:gin}]
Let us prepare with similar steps as in the proof of Theorem \ref{thm:gen} for $p=1$. By smoothing, the Localization Lemma \ref{lem:Localization} for some $R>1$ to be chosen later and Lemma \ref{lem:MESD} for the mean ESD it holds 
\begin{align}\label{eq:preparation}
\IE(W_1(\mu_n,\mu_\infty))&\le \IE(W_1(\mu_n^r,\mu_\infty))+r\nonumber \\
&\le \IE\Big(\sup_{\substack{\mathrm{supp} f\subseteq B_{4R}(0) \\ \mathrm{Lip}(f)\le 1}}\int_{B_{4R}(0)} f(z)d(\mu_n^ r -\bar\mu^r_n)(z)\Big)+2r+\mathcal O (1/n).
\end{align}
Note that in the proof of Corollary \ref{cor:gen}, we have been seeing already that the transport cost on the event where $\mu_n$ is not supported in $B_R(0)$ is negligible. The logarithmic potential enters with \eqref{eq:distrPoisson}
\begin{align}\label{eq:H-2}
\int f(z)d(\mu_n^ r -\bar\mu^r_n)(z)&=-\frac {1} {2\pi}\int f(z)(\Delta U_n^ r(z) -\Delta \bar U_n^r(z))dz\nonumber\\
&\le \frac{1}{2\pi}\int_{B_{4R}(0)} \vert  \nabla f(z)\vert  \vert \nabla (U_n^ r (z) -  \bar U_n^r(z))\vert dz\nonumber\\
& \le \frac{4R}{2\sqrt \pi}\Big(\int \vert \nabla  U_n^ r(z)  -  \nabla \bar U_n^r(z)\vert ^2dz\Big)^{1/2}
\end{align}
where we used integration by parts, then applied Cauchy Schwarz inequality\footnote{\emph{How far can you go with the
Cauchy-Schwarz inequality and integration by parts?} -- First sentence of \cite{bakry}} and $\vert \nabla f\vert \le 1$. Taking the same route back we see that
\begin{align*}
\frac{2R}{\sqrt \pi}\Big(\int \vert \nabla  U_n^ r(z)  -  \nabla \bar U_n^r(z)\vert ^2dz\Big)^{1/2}&=2^{3/2}R\Big(\int (U_n^ r   -  \bar U_n^r )d(\mu_n^r-\bar\mu_n^r)\Big)^{1/2}\\
&=2^{3/2}R\Big(\int (U_n^{2r}   -  \bar U_n^{2r} )d(\mu_n-\bar\mu_n)\Big)^{1/2}
\end{align*}
by \eqref{eq:intconv}. Here we should remark again, that $U_n^{2r}$ is a smooth function with logarithmic growth, hence the integral exists. Together with \eqref{eq:preparation} and Jensen's inequality, it holds
\begin{align}\label{eq:preparationDone}
\IE(W_1(\mu_n,\mu_\infty))\le 2^{3/2}R\Big[\IE\Big(\int ( U_n^{2r}   -  \bar U_n^{2r} )d(\mu_n-\bar\mu_n \Big)\Big]^{1/2} +2r +\mathcal O (1/n).
\end{align}
Using \eqref{eq:intconv} and writing $U_n^{2r}$ as well as $\mu_n$ as sums, we arrive at
\begin{align}\label{eq:arrive}
\IE\Big(\int ( U_n^{2r}   -  \bar U_n^{2r} )d(\mu_n-\bar\mu_n) \Big)&=\IE\Big(\int U_n^{2r} d\mu_n\Big)-\int \bar U_n^{2r}d\bar\mu_n \nonumber\\
&=\frac 1 {n^2} \sum_{j,k=1}^n\IE(-\log^{2r}\vert \lambda_j-\lambda_k\vert )-\int \bar U_n^{2r}d\bar\mu_n\\
&=-\frac 1 n \log^{2r}(0)-\frac{n-1} {n} \IE(\log^{2r}\vert \lambda_1-\lambda_2\vert )-\int \bar U_n^{2r}d\bar\mu_n,\nonumber
\end{align}
where we abbreviated $\log^{2r}\vert \cdot\vert =\log\vert \cdot\vert \star \tilde\phi_r\star\tilde\phi_r$. Rephrase the expectation in terms of the marginal distribution \eqref{eq:marginalGinibre}
\begin{align*}
\frac{n-1} {n} \IE(\log^{2r}\vert \lambda_1-\lambda_2\vert )=&\iint\log^{2r}\vert t-z\vert p^{(1)}(t)p^{(1)}(z)dt dz \\
-\frac 1{\pi^2}&\iint\log^{2r}\vert t-z\vert \exp\big[-n(\vert t\vert ^2+\vert z\vert ^2)\big]\Big\lvert\sum_{k=0}^{n-1}\frac{n^k(t\bar z)^{k}}{k!}\Big\rvert^2 dtdz,
\end{align*}
where $t,z\in\IC$. The first term equals $-\int \bar U_n^{2r}d\bar\mu_n$, cancelling the counterpart of \eqref{eq:arrive}. Therefore, it remains to study
\begin{align}\label{eq:remains}
&\IE\Big(\int ( U_n^{2r}   -  \bar U_n^{2r} )d(\mu_n-\bar\mu_n) \Big)\nonumber\\
=&\frac 1{\pi^2}\iint\log^{2r}\vert t-z\vert \exp\big[-n(\vert t\vert ^2+\vert z\vert ^2)\big]\Big\lvert\sum_{k=0}^{n-1}\frac{n^k(t\bar z)^{k}}{k!}\Big\rvert^2 dtdz-\frac 1 n \log^{2r}(0)
\end{align}
and the goal is to show that it is of order $n^{-1}$.

Let us split the integral into $t\bar z$ being in or outside the unit ball and for the former write out the whole exponential series. Then the whole integral is partitioned into $I+J+K-L-\bar L$, where each part is given by
\begin{align*}
I=&\frac 1{\pi^2}\iint\eins_{B_1(0)^c}(t\bar z) \log^{2r}\vert t-z\vert \exp\big[-n(\vert t\vert ^2+\vert z\vert ^2)\big]\Big\lvert\sum_{k=0}^{n-1}\frac{n^k(t\bar z)^{k}}{k!}\Big\rvert^2 dtdz\\
J=&\frac 1{\pi^2}\iint\eins_{B_1(0)}(t\bar z) \log^{2r}\vert t-z\vert \exp\big[-n(\vert t\vert ^2+\vert z\vert ^2)\big]\exp\big[n(t\bar z+\bar t z)] dtdz\\
K=&\frac 1{\pi^2}\iint\eins_{B_1(0)}(t\bar z) \log^{2r}\vert t-z\vert \exp\big[-n(\vert t\vert ^2+\vert z\vert ^2)\big]\Big\lvert\sum_{k=n}^{\infty}\frac{n^k(t\bar z)^{k}}{k!}\Big\rvert^2 dtdz\\
L=&\frac 1{\pi^2}\iint\eins_{B_1(0)}(t\bar z) \log^{2r}\vert t-z\vert \exp\big[-n(\vert t\vert ^2+\vert z\vert ^2)\big]\exp\big[nt\bar z\big]\sum_{k=n}^{\infty}\frac{n^k(\bar t z)^{k}}{k!}dtdz.
\end{align*}

We will see that $J$ contains the leading order terms and the others are asymptotically negligible as $o(1/n)$. Let $\epsilon:=\sqrt{10\frac{\log n}{n}}$ and define the edge area $\mathsf O:=B_{\sqrt{1+\epsilon}}\setminus B_{\sqrt{1-\epsilon}}$. 

\subsubsection*{The part I}

 %Thus by pulling the absolute value inside, we obtain
%\begin{align}\label{eq:I}
% \vert I\vert \lesssim \log n \iint\eins_{B_1(0)^c}(t\bar z) \exp\big[-n(\vert t\vert ^2+\vert z\vert ^2)\big]\Big\lvert\sum_{k=0}^{n-1}\frac{n^k(t\bar z)^{k}}{k!}\Big\rvert^2 dtdz.
%\end{align}

Again, we have to split the integral depending on different regions. Let $I_1$ be the above integral $I$ restricted to $t\in B_{\sqrt 2}(0)^c$, $I_2$ belongs to $t\in B_{\sqrt 2}(0)\setminus B_{\sqrt{1+\epsilon}}(0)$, $I_3$ for $t\in \mathsf O$ and $I_4$ the remaining part of $t\in B_{\sqrt{1-\epsilon}}(0)$. We will work through these pieces chronologically. 

For the first one, note that $\big\lvert\log ^{2r}\vert t-z\vert \big\rvert\lesssim \log n$ if $r=\kappa/\sqrt n$ and $\vert t-z\vert \le n$. Contrarily if $\vert t-z\vert > n$, then we may assume by symmetry $\vert t\vert \vee\vert z\vert =\vert t\vert >n/2$ without loss of generality (i.e. if $|z|>|t|$, the corresponding integral is bounded by the one for $|t|>|z|$). With this observation, we pull the absolute value inside the integral and apply \eqref{eq:expge1} to obtain
\begin{align*}
I_1&\lesssim\iint_{\vert t\vert ^2\ge 2}\eins_{B_1(0)^c}(t\bar z) \big(\log n\vee \log \vert t\vert \big)\exp\big[-n(\vert t\vert ^2+\vert z\vert ^2)\big]\Big\lvert\sum_{k=0}^{n-1}\frac{n^k(t\bar z)^{k}}{k!}\Big\rvert^2 dtdz\\ 
&\lesssim\frac{n^2}{n}\iint_{\vert t\vert ^2\ge 2}\big(\log n\vee \log \vert t\vert \big)\exp\big[-n(\vert t\vert ^2+\vert z\vert ^2-2-2\log\vert t\bar z\vert )\big] dtdz\\
&=n\int_{\vert t\vert ^2\ge 2}\big(\log n\vee \log \vert t\vert \big)\exp\big[-n(\vert t\vert ^2-1-\log(\vert t\vert ^2))\big]dt\\
&\null\qquad\cdot\int\exp\big[-n(\vert z\vert ^2-1-\log(\vert z\vert ^2))\big] dz
\end{align*}
By Laplace Method\footnote{Of course, one may notice that the whole proof can also be seen as an application of Laplace Method around $\vert t-z\vert =0$.}, the $z$-integral is of order $1/\sqrt n$.  %whose proof we omit since it uses the same estimates that will be carried out below for different parts of the $t$-integral.
Indeed, the exponent $f(\rho)=\rho^2-1-\log(\rho^2)$ has a minimum $f(1)=0$ with curvature $f''(1)=4$ and therefore
\begin{align*}
\int\exp\big[-n(\vert z\vert ^2-1-\log(\vert z\vert ^2))\big] dz=2\pi\int_0^\infty \exp\big[-nf(\rho)\big]\rho d\rho\sim 2\pi\sqrt{\frac{2\pi}{4n}}\lesssim \frac 1 {\sqrt n}.
\end{align*}
In the following, we use the inequality
\begin{align}\label{eq:logineq}
x-1-\log(x)\ge\frac{(x-1)^2}{2}-\frac{(x-1)^3}{3}
\end{align}
for $x>0$, which follows from the Taylor series. Moreover we have for $x-1-\log(x)>ax$ for $x>2$ and some $a>0$ (precisely $a=(1-\log(2))/2$). Consequently, we see that first part $I_1$ is negligible as
\begin{align*}
I_1&\le \sqrt n\log n \int_{\vert t\vert ^2\ge 2}\log\vert t\vert  \exp\big[-an\vert t\vert ^2\big]dt\\
&\lesssim \sqrt n\log n \int_{\vert t\vert ^2\ge 2} \exp\big[-\frac a 2n\vert t\vert ^2\big]dt\lesssim \sqrt n \log n e^{-an}=o(1/n).
\end{align*}

For $I_2$, we restrict ourselves to $\vert t\vert >\vert z\vert $ by symmetry of the integral $I$. From $\big\lvert\log^{2r} \vert t-z\vert \big\rvert\lesssim \log n\vee \log\vert t\vert \le \log n$, \eqref{eq:expge1} and \eqref{eq:logineq} it follows
\begin{align*}
I_2&\lesssim \sqrt n\log n \int_{1+\epsilon\le\vert t\vert ^2\le 2}\exp\Big[-n\Big(\frac{(\vert t\vert ^2-1)^2}{2}-\frac{(\vert t\vert ^2-1)^3}{3}\Big)\Big]dt\\
&\le \sqrt n \log n\int_{1+\epsilon\le\vert t\vert ^2\le 2}\exp\big[-n\frac{(\vert t\vert ^2-1)^2}{6}\big]dt\\
&=\pi\sqrt n\log n \int_\epsilon^1 \exp\big[-n\rho^2/6\big]d\rho\\
&\lesssim \sqrt n \log n e^{-n\epsilon^2/6}=n^{1/2-10/6}\log n=n^{-7/6}\log n
\end{align*}
which is of negligible order.

We turn to $I_3$, the integration close to the edge $\mathsf O$.\footnote{The pointwise asymptotics and zeroes of the exponential sum $\exp[-nw]\sum_{k=0}^{n-1}\frac{(nw)^k}{k!}$ depend very much on the position of $w\in\IC$. For instance if $w$ lies inside the so-called Szeg\"o curve, then we have pointwise convergence, see for instance \cite{szego}. However, integration over whole circles of $w$ will make our analysis independent of these regions.}
Here, we use polar coordinates $z=\rho e^{i\phi}, t=\eta e^{i\theta}$ and rephrase square of the exponential sum as two individual sums to obtain
\begin{align}\label{eq:I3}
\frac{I_3}{\log n }&\le \iint_{t\in\mathsf O}\exp\big[-n(\vert t\vert ^2+\vert z\vert ^2)\big]\Big\lvert\sum_{k=0}^{n-1}\frac{n^k(t\bar z)^{k}}{k!}\Big\rvert^2 dtdz\\
&=\sum_{j,k=0}^{n-1}\int_{\sqrt{1-\epsilon}}^{\sqrt{1+\epsilon}}\int_0^\infty \exp\big[-n(\eta^2+\rho^2)\big]\frac{n^{k+j}\rho ^{k+j+1}\eta ^{k+j+1}}{k!j!}d\rho d\eta\nonumber\\
&\qquad\cdot \int_0^{2\pi}e^{i\theta(k-j)}d\theta\int_0^{2\pi}e^{i\phi(j-k)}d\phi\nonumber\\
&=\frac{2\pi^2}{n}\sum_{k=0}^{n-1}\Big(\frac{1}{k!}\int_{\sqrt{1-\epsilon}}^{\sqrt{1+\epsilon}} e^{-n\eta^2}n^k\eta^{2k+1} d\tilde\eta\Big)\Big(\frac{1}{k!}\int_0^\infty e^{-\tilde\rho}\tilde\rho^k d\tilde\rho\Big)\nonumber\\
&=\frac{2\pi^2}{n}\sum_{k=0}^{n-1}\frac{1}{k!}\int_{\sqrt{1-\epsilon}}^{\sqrt{1+\epsilon}} e^{-n\eta^2}n^k\eta^{2k+1} d\eta\nonumber
\end{align}
where we changed radial variable to $\tilde \rho=n\rho^2$ and used the definition of the Gamma function. The remaining integral is governed by the asymptotics of the incomplete Gamma function, but in our notation of $\bar\mu_n$ (see \eqref{eq:DensityGinibre}) it is possible to phrase it as
\begin{align*}
&\frac 2 n\sum_{k=0}^{n-1}\frac{1}{k!}\int_{\sqrt{1-\epsilon}}^{\sqrt{1+\epsilon}} e^{-n\eta^2}n^k\eta^{2k+1} d\eta=\frac 1 n\bar\mu_n(\mathsf O)\\
&=\frac 1 n \big(\bar\mu_n(B_{\sqrt{1+\epsilon}}(0))-\mu_\infty(B_{\sqrt{1+\epsilon}}(0))+\mu_\infty(B_{\sqrt{1-\epsilon}}(0)) -\bar\mu_n(B_{\sqrt{1-\epsilon}}(0)) +\epsilon\big).
\end{align*}
According to \cite[Lemma 1]{GJ18Rate}, the difference $\bar\mu_n(B)-\mu_\infty(B)$ is of order $n^{-1/2}$ uniformly over balls $B$, hence from $\epsilon=\sqrt{10\frac{\log n}{n}}$ we conclude $\vert I_3\vert \lesssim \frac\epsilon n\log n\lesssim\frac{(\log n)^{3/2}}{n^{3/2}}$.

Estimating $I_4$ works analogously to $I_2$. In particular from $\big\lvert\log^{2r} \vert t-z\vert \big\rvert\lesssim \log n$, \eqref{eq:expge1} and \eqref{eq:logineq} we obtain
\begin{align*}
\frac{I_4}{\log n}&\lesssim \sqrt n \int_{\vert t\vert ^2\le 1-\epsilon}\exp\Big[-n\Big(\frac{(\vert t\vert ^2-1)^2}{2}-\frac{(\vert t\vert ^2-1)^3}{3}\Big)\Big]dt\\
&\le \sqrt n \int_\epsilon^1 \exp\big[-n\rho^2/2\big]d\rho\lesssim \sqrt n e^{-n\epsilon^2/2}=n^{1/2-10/2}=o(n^{-9/2}).
\end{align*}

Combining all parts, we have shown that $\vert I\vert =o(1/n)$ is negligible.

\subsubsection*{The part J}

The integral $J$ will contain the diagonal part $\frac 1 n \log^{2r}(0)$ to be cancelled as well as the leading order term of the whole asymptotic. We begin to rewrite
\begin{align*}
J=&\frac 1{\pi^2}\iint\eins_{B_1(0)}(t\bar z) \log^{2r}\vert t-z\vert \exp\big[-n(\vert t\vert ^2+\vert z\vert ^2)\big]\exp\big[n(t\bar z+\bar t z)] dtdz\\
=&\frac 1{\pi^2}\iint\eins_{B_1(0)}(t\bar z) \log^{2r}\vert t-z\vert \exp\big[-n\vert t-z\vert ^2\big] dtdz.
\end{align*}
Consider the part $\vert t-z\vert >\epsilon$ first, where we can very roughly bound $\big\lvert \log ^{2r}\vert t-z\vert \big\rvert \lesssim \log n \exp[\frac n 2 \vert t-z\vert ^2]$ yielding
\begin{align}\label{eq:2J}
&\iint_{\vert t-z\vert >\epsilon}\eins_{B_1(0)}(t\bar z) \log^{2r}\vert t-z\vert \exp\big[-n \vert t-z\vert ^2\big] dtdz\nonumber\\
&\le 2\log n \iint_{\vert t-z\vert >\epsilon}\eins_{B_1(0)}(z)\exp\big[-\tfrac n 2\vert t-z\vert ^2\big] dtdz
\end{align}
because for $t\bar z\in B_1(0)$ either $t\in B_1(0)$ or $z\in B_1(0)$ and by symmetry of the integral we may assume the latter. After a shift, we have
\begin{align*}
\int_{B_1(0)}\int_{B_\epsilon(0)^c} \exp\big[-\tfrac n 2\vert t\vert ^2\big] dtdz=\frac{\pi^2}n \int_{10\log n}^\infty e^{-\rho/2}d\rho\lesssim n^{-5}
\end{align*}
 and thus it remains to investigate 
 \begin{align}\label{eq:J1}
 &\frac 1{\pi^2}\iint_{\vert t-z\vert <\epsilon}\eins_{B_1(0)}(t\bar z) \log^{2r}\vert t-z\vert \exp\big[-n\vert t-z\vert ^2\big] dtdz\nonumber\\ 
 &= \frac {1}{\pi^2}\int\int_{B_{\epsilon}(0)}\eins_{B_1(0)}((z +w)\bar z)\log^{2r}\vert w\vert  \exp\big[-n\vert w\vert ^2\big] dwdz.
 \end{align}
 
 In order to replace $\eins_{B_1(0)}((z +w)\bar z)$ by $\eins_{B_1(0)}(z)$, note that $\vert z +w\vert \le 1+\epsilon$ for all $z\in B_1(0)$ and if $\vert z\vert <1-\epsilon$, then $(z +w)\bar z\in B_1(0)$. Thus $\vert \eins_{B_1(0)}((z +w)\bar z)-\eins_{B_1(0)}(z)\vert \le \eins_{B_{1+\epsilon}\setminus B_{1-\epsilon}}(z)$ uniformly for $w\in B_\epsilon(0)$, implies
\begin{align*}
&\Big\lvert\frac {1}{\pi^2}\int\int_{B_{\epsilon}(0)}\big(\eins_{B_1(0)}((z +w)\bar z)-\eins_{B_1(0)}(z)\big)\log^{2r}\vert w\vert \exp\big[-n(\vert w\vert ^2)\big] dwdz\Big\rvert\\
&\lesssim \epsilon\log n  \int_{B_\epsilon(0)}\exp[-n\vert w\vert ^2]dw=\frac{\pi\log(n)^{3/2}}{n^{3/2}}\int_0^{10\log n} e^{-\rho}d\rho= o(n^{-1}).
\end{align*}

Moreover we need to replace the smooth cutoff $\log^{2r}\vert \cdot\vert $ by the cutoff \[\log^{\star 2r}\vert \cdot\vert :=\log\vert \cdot\vert \star \frac{1}{2\pi r} \eins_{\partial B_r(0)}\star \frac{1}{2\pi r} \eins_{\partial B_r(0)}.\]
Their difference is negligible as well since $\log^{2r}\vert \cdot\vert =\log^{\star 2r}\vert \cdot\vert \star \phi_\delta\star \phi_\delta$ is a smoothing on scale $\delta=n^{-10}$ and $\log^{\star 2r}$ is $1/r$-Lipschitz, more precisely we have
\begin{align*}
 &\frac {1}{\pi}\int_{B_{\epsilon}(0)}\big(\log^{2r}\vert w\vert -\log^{\star 2r}\vert w\vert \big)\exp\big[-n\vert w\vert ^2\big] dw\\
 \le&  \frac {1}{\pi}\int_{B_{\epsilon}(0)}\Big( \frac{\delta} r \int|z|\, (\phi\star \phi)(z)dz \Big)\exp\big[-n\vert w\vert ^2\big] dw
 \lesssim \frac{n^{-10}}{r}.
\end{align*}

Finally, we arrive at the point to evaluate the leading order term of $J$
\begin{align}\label{eq:leading}
&\frac {1}{\pi^2}\int_{B_1(0)}\int_{B_{\epsilon}(0)}\log^{2r}\vert w\vert \exp\big[-n\vert w\vert ^2\big] dwdz\nonumber\\
&=\frac {1}{\pi}\int_{B_{\epsilon}(0)}\log^{\star 2r}\vert w\vert \exp\big[-n\vert w\vert ^2\big] dw+o\big(\tfrac 1n\big).
\end{align}
Replacing the cutoff $\log ^{\star 2r}$ by the usual logarithm, we would get
\begin{align*}
 \frac {1}{\pi}\int_{B_{\epsilon}(0)}\log\vert w\vert \exp\big[-n\vert w\vert ^2\big] dw&=\frac 1 {2n} \int_0^{10\log n}\log\big(\frac \rho n\big)e^{-\rho}d\rho\\
 &=-\frac 1 {2n}\log n -\frac{\gamma}{2n}+o\big(\tfrac 1 n\big),
\end{align*}
where $\gamma\approx 0.577$ is the Euler–Mascheroni constant. On the other hand, the last term of \eqref{eq:remains} is given by
\begin{align*}
 -\frac 1 n \log^{2r}(0)= -\frac 1 n \log^{\star 2r}(0)=-\frac 1 n \log r=\frac 1{2n} \log n -\frac 1 n \log {\kappa}
\end{align*}
by our explicit choice of the regularized logarithm and for $r=\kappa/\sqrt n$. Analogously to \eqref{eq:Unr} we have $\log ^{\star 2r}\vert w\vert =\log\vert w\vert $ for $w\in B_{2r}(0)^c$ by the mean value property of $\log \vert w\vert $ and $\log(r)\vee\log\vert w\vert \le\log^{\star 2r}\vert w\vert \le\log (2r)<0$ for $w\in B_{2r}(0)$. Thus, we may replace $\log ^{\star 2r}$ by $\log$ (for which we know the explicit value), if we account for the deviation on $B_{2r}(0)\subseteq B_\epsilon(0)$ as follows
\begin{align}
&\frac {1}{\pi}\int_{B_{\epsilon}(0)}\log^{\star 2r}\vert w\vert \exp\big[-n\vert w\vert ^2\big] dw-\frac 1 n \log^{2r}(0)\nonumber\\
&= \frac {1}{\pi}\int_{B_{2r}(0)}\big(\log^{\star 2r}\vert w\vert -\log\vert w\vert \big)\exp\big[-n\vert w\vert ^2\big] dw-\frac 1 n \log{\kappa} -\frac \gamma {2n} +o\big(\tfrac 1 n\big)\nonumber\\
&\le \frac {1}{\pi}\int_{B_{2r}(0)}\log\Big(\frac{2r}{\vert w\vert }\Big)\exp\big[-n\vert w\vert ^2\big] dw-\frac 1 n \log{\kappa} -\frac \gamma {2n} +o\big(\tfrac 1 n\big)\nonumber\\
&= \frac {1}{2n}\int_0^{4\kappa^2}\log\Big(\frac{4\kappa^2}{\rho}\Big)e^{-\rho} d\rho-\frac 1{2n} \log(\kappa^2) -\frac \gamma {2n} +o\big(\tfrac 1 n\big).\nonumber
\end{align}
This can be computed explicitly via integration by parts
\begin{align}\label{eq:mainterm}
&\frac {1}{2n}\int_{4\kappa^2}^{\infty}\log {\rho}e^{-\rho} d\rho+\frac {\log {4}}{2n}\int_0^{4\kappa^2}e^{-\rho} d\rho-\frac {\log (\kappa^2)}{2n}\int_{4\kappa^2}^{\infty}e^{-\rho} d\rho +o\big(\tfrac 1 n\big)\nonumber\\
&= \frac {1}{2n}\big(  e^{-4\kappa^2}\log(4\kappa^2)+\int_{4\kappa^2}^\infty \tfrac 1 \rho e^{-\rho}d\rho +\log 4 (1-e^{-4\kappa^2})-\log(\kappa^2) e^{-4\kappa^2}\big) +o\big(\tfrac 1 n\big)\nonumber\\
&= \frac {\Gamma(0,4\kappa^2)+\log 4}{2n} +o\big(\tfrac 1 n\big),
\end{align}
where we used the definition of the incomplete $\Gamma$ function $\Gamma(0,x)=\int_{x}^\infty \tfrac 1 \rho e^{-\rho}d\rho$.

\subsubsection*{The part K}
This part of the integral is handled similarly to $I$, if not even more simple since by symmetry we again assume $t\in B_1(0)$. Therefore
\begin{align*}
\vert K\vert \lesssim \log n \iint_{\abs t \le 1}\eins_{B_1(0)}(t\bar z)\exp\big[-n(\vert t\vert ^2+\vert z\vert ^2)\big]\Big\lvert\sum_{k=n}^{\infty}\frac{n^k(t\bar z)^{k}}{k!}\Big\rvert^2 dtdz,
\end{align*}
which is split into $K_3$ for $1-\epsilon\le \vert t\vert ^2\le 1$ and $K_4$ for $\vert t\vert ^2\le 1-\epsilon$  (using the same subscripts as for $I$).

For $K_3$ we use polar coordinates $z=\rho e^{i\phi}, t=\eta e^{i\theta}$, rephrase square of the exponential sum as two individual sums and apply dominated convergence to obtain
\begin{align*}
K_3&\le \iint_{\vert t\vert ^2\ge 1-\epsilon}\exp\big[-n(\vert t\vert ^2+\vert z\vert ^2)\big]\Big\lvert\sum_{k=n}^{\infty}\frac{n^k(t\bar z)^{k}}{k!}\Big\rvert^2 dtdz\\
&=\sum_{j,k=n}^{\infty}\int_{\sqrt{1-\epsilon}}^{1}\int_0^\infty \exp\big[-n(\eta^2+\rho^2)\big]\frac{n^{k+j}\rho ^{k+j+1}\eta ^{k+j+1}}{k!j!}d\rho d\eta \\
&\qquad \cdot\int_0^{2\pi}e^{i\theta(k-j)}d\theta\int_0^{2\pi}e^{i\phi(j-k)}d\phi\\
&=\frac{2\pi^2}{n}\sum_{k=n}^{\infty}\Big(\frac{1}{k!}\int_{\sqrt{1-\epsilon}}^1 e^{-n\eta^2}n^k\eta^{2k+1} d\tilde\eta\Big)\Big(\frac{1}{k!}\int_0^\infty e^{-\tilde\rho}\tilde\rho^k d\tilde\rho\Big)\\
&=\frac \pi n\int_{\sqrt{1-\epsilon}}^1 e^{-n\eta^2}\sum_{k=n}^{\infty}\frac{n^k\eta^{2k}}{k!} 2\pi\eta d\eta\lesssim \frac \epsilon n
\end{align*}
where we followed the lines of \eqref{eq:I3} and below. Moreover the bounds \eqref{eq:exple1} and \eqref{eq:logineq} imply
\begin{align*}
K_4&\lesssim \sqrt n \int_{\vert t\vert ^2\le 1-\epsilon}\exp\Big[-n\Big(\frac{(\vert t\vert ^2-1)^2}{2}-\frac{(\vert t\vert ^2-1)^3}{3}\Big)\Big]dt\\
&\le \sqrt n \int_\epsilon^1 \exp\big[-n\rho^2/2\big]d\rho\lesssim \sqrt n e^{-n\epsilon^2/2}=n^{1/2-10/2}=o(n^{-9/2}).
\end{align*}
\subsubsection*{The part L}
For the mixed term $L$ we write
\[\vert L\vert \lesssim \log n\iint h_1(t,z)\overline{h_2(t,z)}dtdz,\]
where we defined 
\begin{align*}
h_1(t,z)&=\eins_{B_1(0)}(t\bar z) \exp\big[-\frac n 2(\vert t\vert ^2+\vert z\vert ^2)\big]\exp\big[nt\bar z\big],\\
h_2(t,z)&=\eins_{B_1(0)}(t\bar z) \exp\big[-\frac n 2(\vert t\vert ^2+\vert z\vert ^2)\big]\sum_{k=n}^{\infty}\frac{n^k(t \bar z)^{k}}{k!}.
\end{align*}
One last time we are going to apply Cauchy Schwarz, such that
\[\vert L\vert \lesssim \log n \Big(\iint \vert h_1(t,z)\vert ^2dtdz\Big)^{1/2}\Big(\iint \vert h_2(t,z)\vert ^2dtdz\Big)^{1/2}\]
and where the first factor corresponds to the integral $J$ and the second to $K$. Note the missing term $\log ^{2r}\vert t-z\vert =\mathcal O (\log n)$ does not influence the following, instead its absence may only improve the asymptotic. We already have seen $J\lesssim n ^{-1}$ and $K\lesssim (\log n)^2 n^{-3/2}$. Therefore we conclude $L\lesssim (\log n)^3 n^{-5/4}$ and hence $L$ is negligible as well.

\subsubsection*{The part, where we put all parts together}

Combining \eqref{eq:preparationDone}, \eqref{eq:remains} and all the parts $I$-$L$ (with the main term being \eqref{eq:mainterm}), we have shown
\begin{align*}
\IE(W_1(\mu_n,\mu_\infty))&\le 2^{3/2}R\Big[I+J+K-L-\overline L -\frac 1 n \log ^{2r}(0)\Big]^{1/2} +2r+\mathcal O (1/n)\\
&\le \Big(2\kappa+\sqrt{4(\log 4+\Gamma(0,4\kappa^2)}R\Big)\frac{1}{\sqrt n}+o(n^{-1/2}).
\end{align*}
Ultimately the claim follows from optimizing $\kappa\approx 1   /4$ and $R\approx1.01$, which even yields a constant $\approx 3.65$.
\end{proof}

\begin{bem}
We believe that the constant can be improved significantly by improving the analysis in \eqref{eq:mainterm} in such a way, that sending $\kappa\to 0$ is possible. This was done in \cite[Proof of Theorem 4.4]{AST}, where a smoothing with the heat kernel at scale $\kappa \sqrt{\log n / n}$ was used ($\gamma$ in their notation). However, finding the exact constant and a corresponding lower bound seems to be very difficult.
\end{bem}

 \begin{bem}
 In both proofs (see for instance \eqref{eq:H-1} and \eqref{eq:H-2}), we basically have used that for absolutely continuous distributions $\mu,\nu$ with compact support it holds
 \[W_1(\mu,\nu)^2\le c \int \vert \nabla U(z)\vert ^2dz.\]
for $\Delta U =-2\pi(\mu-\nu)$ on $\IC$.
A similar bound holds for any compact manifold $M$ and a solution $f$ to the Poisson equation with Neumann boundary conditions
\begin{align*}
 \Delta f =\mu-\nu\quad &\text{, in }M\\
 \nabla f \cdot \mathfrak n_M =0\quad &\text{, on }\partial M.
\end{align*}
Then, \cite[Proposition 2.3]{AST} (see also \cite[Corollary 3]{Peyre}) states
\begin{align*}
W_2(\mu,\nu)^2\le c \int \vert \nabla f(z)\vert ^2dz,
\end{align*}
where the constant however crucially depends on (and blows up with) lower bounds of the densities of $\mu$ and $\nu$.\footnote{One may also view the right hand side as a dual norm $\lVert \mu-\nu \rVert_{\dot H^{-1}}$ of the homogeneous Sobolev space $\dot H^1$.} % has norm $\lVert f\rVert_{\dot H^1} =\big(\int \vert \nabla f\vert ^2 dz\big)^{1/2}$ and its dual norm is $\lVert \mu \rVert_{\dot H^{-1}}:=\sup_{\lVert f \rVert_{\dot H ^1}\le 1}\{\int f d\mu\}$. 
The logarithmic potential $U$ can be seen as a solution $f$ satisfying the Neumann boundary conditions at infinity.

In order to tackle $p$-Wasserstein distances for general $p> 1$ it is tempting to use such inequalities (see \cite[Theorem 2]{Ledoux}) or modify them to logarithmic potentials. Nevertheless, lower bounds on the densities (e.g. of $\mu_n, \bar\mu_n, \mu_\infty$) are impossible to hold on the whole complex plane and if we would restrict ourselves to bounded regions, uncontrollable boundary terms of the logarithmic potential $U$ will appear. Alternatively, one may work with abstract solutions $f$, but then one loses information like Girko's Hermitization Trick \eqref{eq:GirkoHerm}, which is necessary for Theorem \ref{thm:gen}, and the explicit structure of our integral \eqref{eq:remains} disappears. Moreover, for $p>2$ it is unclear how to connect $\int \vert \nabla U\vert ^pdz$ to $\int U d(\mu-\nu)$ like we did in \eqref{eq:ibp=2}. Therefore, we will present a different idea in the next chapter for $p\ge 2$.
\end{bem}

 \section{The proof of Theorem \ref{thm:gen} for $p\ge 2$}

 In order to prove Theorem \ref{thm:gen} for $p\ge 2$, we will relate the Wasserstein metric to the uniform distance from \cite{GJ18Rate, Jalowy} by the following inequality due to Fournier, Guillin \cite{FG15}, which in turn relies heavily on \cite{Dereich13}. Let us introduce some notation first. For $l\ge 0$, we denote by $\mathcal P_l$ the natural partition of $K_0=(-1,1]^2$ into $4^{l}$ translations of $(-2^{-l},2^{-l}]^2$. For $p>0$ and two distributions on $(-1,1]^2$, we introduce the distance
 \[\tilde{\mathcal D}_p(\mu,\nu):=\frac{2^p-1}2\sum_{l=1}^\infty 2^{-pl}\sum_{F\in\mathcal P_l}\vert \mu(F)-\nu(F)\vert ,\]
 which is bounded by 1. More generally define $K_k=(-2^k,2^k]^2\setminus(-2^{k-1},2^{k-1}]^2$. For $k\in\IN$ and a distribution $\mu$ on $\IC$, define $\mathcal R_{K_k}\mu$ to be the pushforward of the conditioned measure $\mu \vert _{K_k}/\mu(K_k)$ under the map $x\mapsto x/2^k$. Define the distance on the probability measures on $\IC$ by
 \begin{align}\label{eq:D}
  \mathcal D_p (\mu,\nu):=\sum_{k=0} ^\infty 2^{pk}\Big(\vert \mu(K_k)-\nu(K_k)\vert +(\mu(K_k)\wedge\nu(K_k))\tilde{\mathcal D}_p(\mathcal R_{K_k}\mu,\mathcal R_{K_k}\nu)\Big).
 \end{align}

 \begin{prop}\label{prop:Ineq}
For $p\ge 1$ and two distributions $\mu, \nu$ on $\IC$ it holds 
\[W_p(\mu,\nu)\lesssim \big(\mathcal D_p(\mu, \nu)\big)^{1/p}\]
and the implicit constant is given by $4(\frac{2^p+1}{2^p-1})^{1/p}$.
 \end{prop}
\begin{proof} This is \cite[Lemma 5]{FG15} for $d=2$, crucially relying on \cite[Lemma 2]{Dereich13}. \end{proof}

Furthermore, define the Kolmogorov distance of two probability measures $\mu, \nu\in\mathcal P (\IC)$
\begin{align}\label{eq:defkolm}
D(\mu,\nu):=\sup_K\vert \mu(K)-\nu(K)\vert ,\end{align}
where the supremum runs over all boxes $K=(a_1,b_1]+i(a_2,b_2]$ for some $a_1,a_2,b_1,b_2\in\IR$. For the measures $\mu=\mu_n, \nu=\mu_\infty$, we have the following uniform bound on the rate of convergence in Kolmogorov distance.
\begin{prop}\label{prop:rateKolm}
 Under the conditions of Theorem \ref{thm:gen}, for any (small) $\epsilon>0$ and every (large) $Q>0$ 
\[\IP(D(\mu_n,\mu_\infty)\le n^{-1/2+\epsilon})\ge1-n^{-Q}\]
holds for $n$ sufficiently large.
\end{prop}
\begin{proof}
This has been proven in \cite[Theorem 2.2]{GJ18Rate} for balls $B$ instead of boxes. There, the choice of balls was more natural with respect to the circular law being supported on a ball. However, the proof remains the same for boxes (and other shapes, e.g. convex sets). We may also derive Proposition \ref{prop:rateKolm} directly from the \cite[Theorem 2.3]{GJ18Rate} for the classical 2-dimensional Kolmogorov distance, simply by recreating an arbitrary box from the difference and union of four half-infinite boxes appearing in the 2-dimensional distribution functions.
\end{proof}

In order to prove the Wasserstein convergence rate, we will simply combine both previous propositions.

\begin{proof}[Proof of Theorem \ref{thm:gen} for $p\ge 2$]
By Hölder's inequality we have $W_q\le W_p$ for $q<p$, hence it suffices to consider $p>2$, then $p=2$ follows from $p=2+\epsilon/2$ because of the negligible error term $n^{\epsilon}$. First note that we may restrict ourselves to a bounded region, say $(-2,2]^2$, since the largest eigenvalue is smaller than $2$ with overwhelming probability (it even converges to $1$ see \cite{Geman, BY86, AEK}).
Thus, by Proposition \ref{prop:Ineq}, we need to estimate
\begin{align*} 
W_p(\mu_n,\mu_\infty)^p\lesssim\mathcal D _p(\mu_n,\mu_\infty)=\mu_n(K_0)\tilde{\mathcal D}_p(\mu_n\vert _{K_0}/\mu_n(K_0), \mu_\infty)+(1+2^p)\mu_n(K_1),
\end{align*}
which follows immediately from the definition \eqref{eq:D}, where only the first two summands matter, since we may assume the supports to be in $[-2,2]^2$. Moreover by the uniform bound of \eqref{eq:defkolm} and Proposition \ref{prop:rateKolm}, we have $\mu_n(K_1)\le D(\mu_n,\mu_\infty)\le n^{-1/2+\epsilon}$ w.o.p. and similarly $\mu_n(K_0)=1+\mathcal O (n^{-1/2+\epsilon})$. From Proposition \ref{prop:rateKolm} it follows that with overwhelming probability 
\begin{align*}
\mu_n(K_0)\tilde{\mathcal D}_p(\mu_n\vert _{K_0}/\mu_n(K_0),\mu_\infty)=&
\frac{2^p-1}2\sum_{l=1}^\infty 2^{-pl}\sum_{F\in\mathcal P_l}\lvert\mu_n (F)-\mu_\infty(F)\mu_n(K_0)\rvert \\
\lesssim &\sum_{l=1}^\infty 2^{-pl}\left(n^{-1/2+\epsilon}+\sum_{F\in\mathcal P_l}\lvert\mu_n (F)-\mu_\infty(F)\rvert\right)\\
\lesssim & \sum_{l=1}^\infty 2^{-(p-2)l}n^{-1/2+\epsilon}\lesssim n^{-1/2+\epsilon} 
\end{align*}
if $p>2$ and the claim follows.
\end{proof}

Note that the proof of the $W_1$ rate of convergence in \cite{OW} as well as the proof of Proposition \ref{prop:Ineq} relies on an explicit construction of a transport map on small boxes. In order to reach the nearly optimal rate, we suggested an inexplicit approach via duality in the preceding sections.
 
\section*{Acknowledgements}
I would like to thank Anna Gusakova, Martin Huesmann and Matthias Erbar for helpful discussions and valuable suggestions. Furthermore, I thank the referees for reading the manuscript so thoroughly and for all their great feedback.

%\bibliography{bibliography_W}%copy bib into same folder! and comment bibliography and biblatex from the beginning!
%\bibliographystyle{alpha} 
%%%%%%%%%\printbibliography

\end{document}